\documentclass[12pt]{article}
\usepackage{amsmath,amsthm,amssymb}
\allowdisplaybreaks
\usepackage{mathrsfs}
\usepackage{cite}
\newtheorem{theorem}{Theorem}[section]
\newtheorem{corollary}[theorem]{Corollary}
\newtheorem{lemma}[theorem]{Lemma}
\newtheorem{proposition}[theorem]{Proposition}
\theoremstyle{remark}
\newtheorem{definition}[theorem]{Definition}
\theoremstyle{remark}
\newtheorem{example}[theorem]{Example}
\theoremstyle{remark}
\newtheorem{remark}[theorem]{Remark}

\newcommand{\cd}{\circledast}
\newcommand{\h}{\mathcal{H}}
\newcommand{\hc}{\mathcal{H}_{\mathbb{C}}}
\newcommand{\fh}{\mathcal F^Q(\mathcal H)}

\newcommand{\wick}[1]{{:}\,#1\,{:}}
\newcommand{\Wick}[1]{{:}\,#1\,{:}_W}

\newcommand{\rom}[1]{{\rm #1}}
\newcommand{\la}{\langle}
\newcommand{\ra}{\rangle}

\begin{document}

\makeatletter\@addtoreset{equation}{section}

\begin{center}{\Large \bf
  Noncommutative L\'evy processes for generalized (particularly anyon) statistics
}\end{center}

{\large Marek Bo\.zejko}\\
Instytut Matematyczny, Uniwersytet Wroc{\l}awski, Pl.\ Grunwaldzki 2/4, 50-384 Wroc{\l}aw, Poland\\
e-mail: \texttt{bozejko@math.uni.wroc.pl}\vspace{2mm}

{\large Eugene Lytvynov}\\ Department of Mathematics,
Swansea University, Singleton Park, Swansea SA2 8PP, U.K.\\
e-mail: \texttt{e.lytvynov@swansea.ac.uk}\vspace{2mm}

{\large Janusz Wysocza\'nski}\\
Instytut Matematyczny, Uniwersytet Wroc{\l}awski, Pl.\ Grunwaldzki 2/4, 50-384 Wroc{\l}aw, Poland\\
e-mail: \texttt{jwys@math.uni.wroc.pl}\vspace{2mm}

{\small

\begin{center}
{\bf Abstract}
\end{center}
\noindent
Let $T=\mathbb R^d$. Let a function $Q:T^2\to\mathbb C$ satisfy $Q(s,t)=\overline{Q(t,s)}$ and $|Q(s,t)|=1$. A generalized statistics is described by creation operators $\partial_t^\dag$ and annihilation operators $\partial_t$, $t\in T$, which satisfy the $Q$-commutation relations:
$\partial_s\partial^\dag_t = Q(s, t)\partial^\dag_t\partial_s+\delta(s, t)$, 
$\partial_s\partial_t = Q(t,s)\partial_t\partial_s$, $\partial^\dag_s\partial^\dag_t = Q(t, s)\partial^\dag_t\partial^\dag_s$. From the point of view of physics, the most important case of a generalized statistics is the anyon statistics, for which  $Q(s,t)$ is equal to $q$ if $s<t$, and to $\bar q$ if  $s>t$. Here $q\in\mathbb C$, $|q|=1$.  We start the paper with a detailed discussion of a $Q$-Fock space  and operators $(\partial_t^\dag,\partial_t)_{t\in T}$ in it, which satisfy the $Q$-commutation relations. 
Next,  we consider a noncommutative stochastic process (white noise)
$\omega(t)=\partial_t^\dag+\partial_t+\lambda\partial_t^\dag\partial_t$, $t\in T$. Here $\lambda\in\mathbb R$ is a fixed parameter. The case $\lambda=0$ corresponds to a $Q$-analog of Brownian motion, while  $\lambda\ne0$ corresponds to a (centered) $Q$-Poisson process. We study $Q$-Hermite ($Q$-Charlier respectively) polynomials of  infinitely many noncommutatative variables $(\omega(t))_{t\in T}$. The main aim of the paper is to explain the notion of independence for a generalized statistics, and to derive corresponding L\'evy processes. To this end, we recursively define $Q$-cumulants of a field $(\xi(t))_{t\in T}$. This allows us to define a $Q$-L\'evy process as a field $(\xi(t))_{t\in T}$ whose values at different points of $T$ are $Q$-independent and which possesses a stationarity of increments (in a certain sense).  We present an explicit construction of a $Q$-L\'evy process, and derive a Nualart--Schoutens-type chaotic decomposition for such a  process.  

 } \vspace{2mm}

\section{Introduction} A first rigorous interpolation between  canonical commutation relations (CCR) and canonical anticommutation relations (CAR) was constructed  in 1991 by Bo\.zejko and Speicher \cite{BS}.  Given a Hilbert space $\mathcal H$, they constructed, for each  $q\in(-1,1)$, a deformation of the full Fock space over $\mathcal H$, denoted by $\mathcal F^q(\mathcal H)$.
For each $h\in\mathcal H$, one naturally defines a (bounded) creation operator, $a^+(h)$, in $\mathcal F^q(\mathcal H)$. The corresponding annihilation operator, $a^-(h)$, is the adjoint of $a^+(h)$.  These operators satisfy the $q$-commutation relations:
\begin{equation}\label{dtrs54t}a^-(g)a^+(h)-qa^+(h)a^-(g)=(g,h)_{\mathcal H},\quad g,h\in\mathcal H.\end{equation}
The limiting cases, $q=1$ and $q=-1$, correspond to the bose and fermi statistics, respectively. It should be stressed that, for $q\ne\pm1$, the $q$-modification of the (anti)symmetrization operator  is a strictly positive operator. Therefore, unlike in the classical bose and fermi cases, there are no commutation relations between the creation operators. A noncommutative analog of Brownian motion (Gaussian process) is the family of operators, $(a^+(h)+a^-(h))_{h\in\mathcal H}\,$, in $\mathcal F^q(\mathcal H)$. A study of  this noncommutative stochastic process was initiated in 
 \cite{BKS,BS,BS_1996}, for further results and generalizations of a noncommutative Brownian motion, see 
 e.g.\   \cite{ABBL,B1,B3,BY,KW}. 

After  \cite{BS},  a series of papers \cite{Biane,BS_1994,JSW,Krolak,Krolak2,L-P,S_1993} appeared, which studied   {\em discrete} generalizations of the 
$q$-commutation relations. In the most general form, such commutation relations look as follows. Let $T$ be a discrete set, and let $\mathcal H$ be the complex space $\ell_2(T)$. Fix a bounded linear operator $\Psi$ acting on $\mathcal H\otimes\mathcal H$ which satisfies the following conditions: $\Psi$ is self-adjoint; the norm of $\Psi$ is $\le1$;  $\Psi$ satisfies the braid relation, see \cite{BS_1994} for details.  Let $(e_i)_{i\ge1}$ be the standard orthonormal basis of $\mathcal H=\ell_2(T)$. Define numbers $q^{ik}_{jl}$ through $\Psi e_j\otimes e_l=\sum_{i,k}q^{ik}_{jl} e_i\otimes e_k$. Then, by \cite{BS_1994}, one can construct a unique Fock representation of the commutation relations
\begin{equation}\label{iftru8rf} a_i^-a_j^+-\sum_{k,l}q_{jl}^{ik}a_k^+a_l^-=\delta_{ij},\quad i,j\in T,\end{equation}
where $(a_i^+)^*=a_i^-$. It should be noted that, in majority of the the above cited papers, main attention is drawn to the case where the norm of the operator $\Psi$ is strictly less than~1.

Another generalization of  the CCR and CAR was proposed in 1995 by Ligouri and Mintchev \cite{LM2,LM}. They fixed a {\it continuous} underlying space $T=\mathbb R^d$ and considered a function $Q:T^2\to\mathbb C$ satisfying $Q(s,t)=\overline{Q(t,s)}$ and $|Q(s,t)|=1$. Setting $\mathcal H$ to be the complex space $L^2(T)$, one defines a bounded linear operator $\Psi$    acting on $\mathcal H\otimes\mathcal H$ by the formula
\begin{equation}\label{ftyde6e7} \Psi(f\otimes g)(s,t)=Q(s,t)g(s)f(t),\quad f,g\in \mathcal H.\end{equation}
This operator is self-adjoint, its norm is equal to~1, and it satisfies the braid relation. One then defines  corresponding creation and annihilation operators, $a^+(h)$ and $a^-(h)$, for $h\in\mathcal H$. By setting $a^+(h)=\int_{T}dt\, h(t)\partial_t^\dag$ and $a^-(h)=\int_{T}dt\, \overline{h(t)}\partial_t$, one gets (at least  informally) creation and annihilation operators, $\partial_t^\dag$ and $\partial_t$, at point $t\in T$. These operators satisfy the $Q$-commutation relations
\begin{gather}
\partial_s\partial^\dag_t - Q(s, t)\partial^\dag_t\partial_s=\delta(s, t),\notag\\
\partial_s\partial_t - Q(t,s)\partial_t\partial_s=0 ,\quad 
\partial^\dag_s\partial^\dag_t - Q(t, s)\partial^\dag_t\partial^\dag_s=0 .\label{lkifude}
\end{gather}
Compared with \eqref{dtrs54t} and  \eqref{iftru8rf}, formula \eqref{lkifude} contains commutation relations between the creation operators, and hence also between the annihilation operators. This is due to the fact that each $n$-particle subspace of the corresponding $Q$-Fock space, $\mathcal F^Q(\mathcal H)$, consists of $Q$-symmetric functions. In particular, such functions are completely determined by their values on the Weyl chamber, i.e., on  the set where $t_1<t_2<\dots<t_n$. (We discuss below how an ordering can be introduced if the dimension $d$ of the underlying space $T=\mathbb R^d$ is $\ge2$.)

From the point of view of physics, the most important case of a generalized statistics  \eqref{lkifude} is the anyon statistics, see e.g.\ the recent physical review papers \cite{NSSFDS,Stern}. For the anyon statistics, the function $Q$ is given by 
$$Q(s,t)=\begin{cases}q,&\text{if }s<t,\\
\bar q,&\text{if }s>t\end{cases}$$
for a fixed $q\in\mathbb C$ with $|q|=1$. Hence, the commutation relations \eqref{lkifude} become
 \begin{gather}
\partial_s\partial^\dag_t - q\partial^\dag_t\partial_s=\delta(s, t),\notag\\
\partial_s\partial_t - \bar q\partial_t\partial_s =0,\quad 
\partial^\dag_s\partial^\dag_t - \bar q\partial^\dag_t\partial^\dag_s=0 ,\label{gisyacfcvfy}
\end{gather}
 for $s<t$. In 1995, Goldin and Sharp \cite{Goldin_Sharp} arrived at these commutation relations as a ``consequence of the group representations describing anyons, together with the (completely general) interwinning property of the field.'' Goldin and Sharp \cite{Goldin_Sharp} realized the $(q,\bar q)$-commutation relations  \eqref{gisyacfcvfy} through operators acting on the space of functions of finite configurations in $T=\mathbb R^2$ (this, in fact, corresponds to the (classical) symmetric Fock space over $\mathcal H=L^2(T)$). An equivalent realization of these commutation relations through operators acting on a Fock space of $(q,\bar q)$-symmetric functions was done by Goldin and Majid in \cite{GM}. They also showed that, in the case where $q$ is an $N$-th root of 1, the corresponding statistics  satisfies the natural anyonic exclusion principle, which generalizes  Pauli's exclusion principle for fermions. 
  
Sections \ref{hfhjfut} and \ref{utyr65e} of this paper contain a rather detailed discussion on the construction of the representation of the $Q$-commutation relations \eqref{lkifude}, with a special attention to the case of anyons. While many results in these two sections can be found  in \cite{GM, LM}  (and to some extent in \cite{BS_1994}), Sections \ref{hfhjfut} and \ref{utyr65e} also contain some new results, like an explicit formula for the $Q$-symmetrization operator (Proposition~\ref{cftrst}) or a derivation of a neutral operator, $a^0(h):=\int_Tdt\,h(t) \partial_t^\dag\partial_t$, in the $Q$-Fock space $\mathcal F^Q(\mathcal H)$. For the reader's convenience, we tried to make our presentation  essentially self-contained. 
We hope that these two sections might be useful even to those readers who are not particularly interested in our further  results related to noncommutative probability for generalized statistics. 

Having creation, neutral, and annihilation operators at our disposal, we define and study, in Section~\ref{cfydry}, a noncommutative stochastic process (white noise) $\omega(t)=\partial_t^\dag+\partial_t+\lambda \partial_t^\dag\partial_t$, $t\in T$. Here $\lambda\in\mathbb R$ is a fixed parameter. The case $\lambda=0$ corresponds to a $Q$-analog of Brownian motion, while the case $\lambda\ne0$ (in particular, $\lambda=1$) corresponds to a (centered) $Q$-Poisson process (compare with \cite{BKS,BL1,A1}). We identify corresponding $Q$-Hermite ($Q$-Charlier respectively)  polynomials, denoted by
$\wick{\omega(t_1)\dotsm\omega(t_n)}\,$,  of infinitely many noncommutative variables $(\omega(t))_{t\in T}$. 
As $\omega(t)$ is written in terms of the creation and annihilation operators, $\partial_t^\dag$ and $\partial_t$, we discuss a relation between the orthogonal polynomials and a natural Wick (normal) ordering, compare with \cite{BKS,BL1,Krolak}. It appears that these are {\it different} procedures, unless $\lambda=0$ (Gaussian case) and the function $Q$ is 
real-valued, i.e., taking values in $\{-1,1\}$ (a mixed 
bose-fermi statistics). We also represent a monomial as a sum of orthogonal polynomials (Wick rule for a product of fields). This immediately implies a corresponding moment formula. 

The main aim of this paper is to explain the notion of independence for a generalized statistics, and to derive corresponding L\'evy processes. We know from experience both in free probability and in $q$-deformed probability that a natural way to explain 
that certain noncomutative random variables are independent (relative to a given statistics/deformation of commutation relations) is to do this through corresponding deformed cumulants. Here we refer the reader to Speicher \cite{FreeCumulants} for a relation between cumulants and independence in the framework of free probability, and  to Anshelevich  \cite{A2} for a definition and study of
$q$-deformed cumulants ($-1<q<1$). See also Lehner \cite{Lehner1}, \cite{Lehner3} for a quite general discussion of cumulants in noncommutative probability. 
Noncommutative 
L\'evy processes have most actively been studied in the framework  of free probability, see e.g.\ \cite{BNT} and the references therein.  Using $q$-deformed cumulants, Anshelevich \cite{A1} constructed and studied noncommutative 
L\'evy processes for $q$-commutation relations \eqref{dtrs54t}. One should also mention that noncommutative L\'evy processes have actively  been studied on various algebraic structures,  see e.g.\ \cite{Franz} and the references therein.

 So, in Section~\ref{gtuyr6}, using the moment formula for a generalized statistics as a hint, we introduce $Q$-deformed cumulants. Since the function $Q$ is not a constant, unless $Q$ is identically equal to $+1$ or $-1$ (bosons or fermions), we cannot expect  to have a definition of cumulants for general noncommutative random variables. Instead, we  recursively define $Q$-cumulants of a  {\it field\/} $
 \xi=(\xi(t))_{t\in T}$ (an operator-valued distribution on $T$). The $n$-th $Q$-cumulant,  $C_n(\xi(t_1),\dots,\xi(t_n))$, is  a {\it measure} $c_n(dt_1\times\dots\times dt_n)$ on $T^n$. For test functions $f_1,\dots,f_n$ on $T$, the $n$-th $Q$-cumulant of $\la f_1,\xi\ra,\dots\la f_n,\xi\ra$ is then given by $\int_{T^n}f_1(t_1)\dotsm f_n(t_n)c_n(dt_1\times\dots\times dt_n)$. Here, for a test function $f$ on $T$, $\la f,\xi\ra$ is the operator $\int_Tdt\,f(t)\xi(t)$. Note that, in the classical case, $Q\equiv 1$, our definition of cumulants leads to the classical cumulants, see e.g.\ \cite{Shiryaev}. 
 Having constructed $Q$-cumulants, we can easily explain what it means that noncommutative random variables 
$\la f_1,\xi\ra,\dots,\la f_n,\xi\ra$ are $Q$-independent. This is done by a complete analogy with classical probability (as well as with free probability).

 In Section~\ref{kigf7ur7}, we define a $Q$-L\'evy process as a field $(\xi(t))_{t\in T}$ whose values at different points of the underlying space $T$ are independent and which possesses the `stationarity of increments' (in a certain sense).  We then present an explicit construction of  a $Q$-L\'evy process as a field in a $Q$-Fock space over $L^2(T)\otimes L^2(\mathbb R,\nu)$. Here $\nu$ is a probability measure on $\mathbb R$ and $\tilde\nu(dx):=\chi_{\mathbb R\setminus\{0\}}x^{-2}\,dx$ is the $Q$-L\'evy measure of the process. It is interesting to note that, for a set $\Delta\subset T$ such that $\int_\Delta dt=1$, the $n$-th $Q$-cumulant of the random variable $\int_\Delta dt\, \xi(t)$ is equal to the $n$-th moment of the $Q$-L\'evy measure $\tilde\nu$ (for $n\ge3$), a property which one  would indeed  expect from a proper L\'evy process. We also show that a $Q$-L\'evy process possesses a property of pyramidal independence (e.g.\ \cite{BS_1996}),  and that the vacuum vector is cyclic for a $Q$-L\'evy process.

It is a well known fact of classical probability that, among all L\'evy process, only Brownian motion and Poisson process possess  the chaos decomposition property, i.e., any square-integrable functional of such a process  can be represented as a sum of mutually orthogonal multiple stochastic integrals with respect to the (centered) process, see e.g.\ \cite{Meyer}.
For a general L\'evy process,
Nualart and Schoutens \cite{NS} derived an orthogonal decomposition of any   square-integrable functional of the process in multiple stochastic integrals with respect to the orthogonalized power jump processes (see also \cite{L2}). Anshelevich \cite{A1} extended the result of \cite{NS} to the case of a $q$-L\'evy processes, $-1<q<1$. (It should be noted that, for $q\ne0$,  Proposition~9 of \cite{A1} holds, in fact, in a slightly modified form, which later affects Proposition~16 of \cite{A1}.) In \cite{BL1}, within the framework of free probability ($q=0$), a Nualart--Schoutens-type decomposition for free L\'evy process was applied for a derivation of free Meixner processes.
So, in final Section~\ref{kifr78ufsrs}, we 
 derive a counterpart of the Nualart--Schoutens chaotic decomposition for $Q$-L\'evy processes. We hope that the result of this section will, in particular, be useful for a discussion of noncommutative Meixner processes for a generalized statistics, compare with  \cite{BL1,L2}.

Let us  note that most results of this paper admit a generalization to the case where the complex-valued function $Q(s,t)$, identifying the statistics (see \eqref{ftyde6e7}), is Hermitian and satisfies $|Q(s,t)|\le1$, compare with \cite{BS_1994}. Also, some extensions are possible in the case of a $q$-statistics with $q\in\mathbb R$ and $|q|>1$, cf.\ \cite{B3}.

\section{Symmetrization operator}\label{hfhjfut}

Let $T$ be a locally compact Polish space, let $\mathcal B(T)$ be the Borel $\sigma$-algebra on $T$, and let $\mathcal B_0(T)$ denote the family of  all pre-compact sets from  $\mathcal B(T)$. Let $\sigma$ be a Radon non-atomic measure on $(T,\mathcal B(T))$.
Let $D:=\{(t, t)\in T^2\mid  t\in T \}$ be the diagonal in $T^2$. Since the measure $\sigma$ is non-atomic,  $\sigma^{\otimes 2}(D)=0$.
Consider a set $A\in\mathcal B(T^2)$ which is symmetric, i.e.,  if $( s,t)\in A$ then $(
t,s)\in A$, and such that $D\subset A$ and $\sigma^{\otimes 2}(A)=0$. Note that  the set
$T^{(2)}:=T^2\setminus A$ is also symmetric. We fix a measurable function
\[
Q:T^{(2)} \mapsto S^1:=\{z\in \mathbb{C}\mid |z|=1 \}
\] which is Hermitian:
\[
Q(s, t) = \overline{Q(t, s)},\quad (s, t)\in T^{(2)}.
 \] Note that the function $Q$ is defined $\sigma^{\otimes 2}$-almost everywhere on
$T^2$.
\begin{example}[Anyons] Let us assume that, for a set $A\subset T^2$ as above, we
have a  strict order outside of $A$, i.e.,
 for all
$(s, t)\in T^{(2)}$ either $s<t$ or $t<s$. For a fixed $q\in S^1$, we
define  a function $Q$ on $T^{(2)}$ as follows:
\begin{equation}\label{kdsty}
Q(s, t):= \begin{cases}q, &\text{if }s<t,\\\bar q,&\text{if }t<s.\end{cases}
\end{equation} Here typical choices would be $T=\mathbb{R}$ or $T=\mathbb{R}_+$,
with $A=D$ and the natural order. More examples one gets if, in
$T:=\mathbb{R}^d$, one considers the set
\[
A:=\{(s, t)\in T^2: s_1=t_1 \}
\] for $s=(s_1, \dots , s_d),\, t=(t_1, \dots , t_d)\in
\mathbb{R}^d$, and the order is given by
$$
s<t \mbox{ if and only if }  s_1<t_1
 $$ for $(s, t)\in T^{(2)}$.
Strictly speaking, the case of anyon statistics will correspond to  $d=2$. (See e.g.\ \cite{GM,Goldin_Sharp,NSSFDS,Stern} and the references therein.)
\end{example}

\begin{example}  Let $T$ be a locally compact Polish space and choose any metric, denoted by $\texttt{dist}$,  which generates the topology on $T$. Choose
 $A=D$, and for a given $r>0$,  define  a real-valued function $Q$ by
$$
Q(s, t):=\begin{cases}
1,&\text{if }\texttt{dist}(s,t)\ge r,\\
-1,&\text{if }\texttt{dist}(s,t)<r$$
\end{cases}
$$ for $(s, t)\in
T^{(2)}$. This will later correspond to mixed commutation and anti-commutation relations (compare with e.g.\ \cite{Biane,L-P,S_1993}).
\end{example}

Given a Hermitian function $Q$ as above, we define a $Q$-symmetry as follows.
We  consider an operator $\Psi$ which transforms a measurable function $f^{(2)}:T^{(2)}\to\mathbb C$
into
$$(\Psi f^{(2)})(s,t):=Q(s,t)f(t,s),\quad (s,t)\in T^{(2)}.$$
In particular, a function $f^{(2)}$ is $Q$-symmetric if $\Psi(f^{(2)})=f^{(2)}$, so that
$$ f^{(2)}(s,t)=Q(s,t)f^{(2)}(t,s).$$

By analogy with $T^{(2)}$, we define
$$ T^{(n)}:=\big\{(t_1,\dots,t_n)\in T^n\mid \forall 1\le i<j\le n:\  (t_i,t_{j})\not\in A\big\},\quad n\ge2,$$
and clearly $\sigma^{\otimes n}(T\setminus T^{(n)})=0$.
The operator $\Psi$ can be naturally extended to act on measurable functions $f^{(n)}:T^{(n)}\to\mathbb C$. Indeed, for $j\in\mathbb N$ and for $n\ge j+1$, we set
\begin{equation}\label{operators Psi_j}
(\Psi_j f^{(n)})(t_1, \dots , t_n):= Q(t_j, t_{j+1}) f(t_1, \dots ,
t_{j-1},t_{j+1}, t_{j},t_{j+2}, \dots , t_n )
\end{equation}
for  $(t_1\dots,t_n)\in T^{(n)}$
The following proposition follows directly from \eqref{operators Psi_j}.

\begin{proposition}\label{kjdstey}
The operators $\Psi_j$, $j\in\mathbb N$, satisfy the  equations:
\begin{align}
\Psi_j^2 &=\operatorname{id},\notag\\
\Psi_j\Psi_i  & =  \Psi_i\Psi_j , \quad |i-j|\ge 2,\notag \\
\Psi_j\Psi_{j+1}\Psi_j  & =   \Psi_{j+1}\Psi_{j}\Psi_{j+1} , \label{Yang-Baxter}
\end{align}
the latter equality being called the Yang--Baxter equation. Here $\operatorname{id}$ denotes the identity operator.

\end{proposition}

In what follows we will use the notations:
\[
\h := L^2(T, \sigma), \quad \hc := L^2(T\mapsto
\mathbb{C}, \sigma)
\] for the Hilbert space of real-valued, respectively complex-valued, square  integrable
functions on $T$.
Thus, for each $n\in\mathbb N$, $\hc^{\otimes n}=L^2(T^n\mapsto
\mathbb{C}, \sigma^{\otimes n})$.
For each $j=1,\dots,n-1$, $\Psi_j$ determines a unitary operator in $\hc^{\otimes n}$.
Consider the group $S_n$ of all permutations of $1,\dots,n$. With each transposition $\pi_j:=(j, j+1)\in S_n$, we associate the operator $\Psi_j$ in $\hc^{\otimes n}$. By Proposition~\ref{kjdstey}, this mapping can be multiplicatively extended to a unitary representation of $S_n$ in $\hc^{\otimes n}$, see e.g.\ \cite{BS_1994,CM}.
More explicitly, represent each permutation $\pi\in S_n$ as an arbitrary product of transpositions,
\begin{equation}\label{representation of permutation}
\pi = \pi_{j_1}\dotsm \pi_{j_k},
\end{equation}
and set
\begin{equation}\label{vtsh} \Psi_{\pi}:=\Psi_{j_1}\dotsm \Psi_{j_k}.\end{equation}
Then, the definition of the unitary operator $\Psi_\pi$ does not depend on the representation of $\pi$ as in \eqref{representation of permutation}, and for any $\pi,\rho\in S_n$, $\Psi_\pi\Psi_\rho=\Psi_{\pi\rho}$.
This allows us to define  a $Q$-symmetrization
operator $P_n$   by
\begin{equation}\label{Q-symmetrization}
P_n:= \frac{1}{n!}\sum_{\pi\in S_n}\Psi_{\pi}.
\end{equation}

\begin{proposition}\label{jvfytdy}
For each $n\in\mathbb N$, the operator $P_n$ is an orthogonal projection in $\hc^{\otimes n}$, i.e.,
$P_n^*=P_n=P_n^2$. Furthermore, for  $1\le k\le n-1$, we have
\begin{equation} \label{gfy7rdy} P_n=P_n(P_k\otimes P_{n-k}).\end{equation}
\end{proposition}

\begin{proof} For each $\pi\in S_n$, we clearly have $\Psi_\pi^*=\Psi_\pi^{-1}=\Psi_{\pi^{-1}}$.
Hence, by \eqref{Q-symmetrization}, $P_n^*=P_n$. Next,
$$
P_n^2=\frac1{(n!)^2}\sum_{\rho\in S_n}\sum_{\pi\in S_n}\Psi_{\rho}\Psi_\pi
=\frac1{(n!)^2}\sum_{\rho\in S_n}\sum_{\pi\in S_n}\Psi_{\rho\pi}=\frac1{n!}\sum_{\pi\in S_n}\Psi_\pi =P_n.
$$
Analogously one can also prove formula \eqref{gfy7rdy}.
\end{proof}

Thus, similarly to the symmetric and antisymmetric tensor products, one can naturally define a {\it $Q$-symmetric tensor product}, which will be denoted by $\circledast$. More precisely, we denote
$\mathcal H_{\mathbb C}^{\circledast n}:=P_n\hc^{\otimes n}$, and for any $m,n\in\mathbb N$ and any
$f^{(m)}\in \hc^{\circledast m}$, $g^{(n)}\in \hc^{\circledast n}$, $f^{(m)}\circledast g^{(n)}:=P_{m+n}(f^{(m)}\circledast g^{(n)})$. In particular, for any $f_1,\dots,f_n\in \hc$,
$f_1\circledast\dots \circledast f_n=P_n(f_1\otimes\dots\otimes f_n)$. Note that, by formula \eqref{gfy7rdy},
this tensor product is associative.

We will say that a measurable function $f^{(n)}:T^{(n)}\to\mathbb C$ ($n\ge 2$) is $Q$-symmetric if $\Psi_jf^{(n)}=f^{(n)}$ for all $j=1,\dots,n-1$.
The following trivial proposition shows that, as expected, the Hilbert space $\hc^{\circledast n}$ consists of all $Q$-symmetric functions from $\hc^{\otimes n}$.

\begin{proposition}\label{lhgufr7} For each $n\ge2$, we have
$$\hc^{\circledast n}=\big\{f^{(n)}\in \hc^{\otimes n}\mid \forall j=1,\dots,n-1: \Psi_j f^{(n)}=f^{(n)}\big\}.$$
\end{proposition}

\begin{proof} Assume that $f^{(n)}\in \hc^{\otimes n}$ satisfies $\Psi_j f^{(n)}=f^{(n)}$ for all $j=1,\dots,n-1$. Then, by \eqref{vtsh}, $\Psi_{\pi}f^{(n)}= f^{(n)}$ for all $\pi\in S_n$, and so $P_n f^{(n)}=f^{(n)}$. Therefore, $f^{(n)}\in \hc^{\circledast n}$. On the other hand, assume that $f^{(n)}\in \hc^{\circledast n}$. Then, for each $j=1,\dots,n-1$,
$$
\Psi_j f^{(n)}=\Psi_j P_n f^{(n)}=\Psi_j\,\frac1{n!}\sum_{\pi\in S_n}\Psi_\pi f^{(n)}=\frac1{n!}\sum_{\pi\in S_n}
\Psi_{\pi_j\pi}f^{(n)}=P_nf^{(n)}=f^{(n)}.
$$
\end{proof}

\begin{remark}
By Proposition~\ref{lhgufr7}, any function from $\hc^{\circledast n}$ is completely determined by its values on the set $\{(t_1,\dots,t_n)\in T^{(n)}\mid t_1<t_2<\dots<t_n\}$. 
\end{remark}

We also have the following inductive formula for the projections $P_n$.

\begin{proposition} \label{ydwu54e} For each $n\in \mathbb N$
\begin{equation}\label{jhfydfujuf} P_{n+1}=\frac1{n+1}(\mathbf 1+\Psi_1+\Psi_2\Psi_1+\dots+\Psi_n\Psi_{n-1}\dotsm \Psi_1)(\mathbf1\otimes P_n),\end{equation}
or equivalently, for any  $h\in \hc$ and $f^{(n)} \in \hc^{\cd n}$ we
have
\begin{multline}
(h\circledast  f^{(n)})(t_1, \dots , t_n, t_{n+1}) =\frac{1}{n+1}\bigg[ h(t_1)f^{(n)}(t_2, \dots , t_{n+1})\\
\text{} +
\sum_{k=2}^{n+1} Q(t_1, t_k) Q(t_2, t_k) \dotsm Q(t_{k-1}, t_{k})
h(t_k)f^{(n)}(t_1, \dots, \check t_k, \dots , t_{n+1})\bigg],\label{dstyufut}
\end{multline}
where $\check t_k$ denotes the absence of $t_k$.
\end{proposition}

\begin{proof} Such a statement is well known in the theory of permutation groups and is, in fact, based on the geometry of the Cayley graph, see e.g.\ \cite{Knuth}.

For each permutation $\sigma\in S_n$, denote by $\mathbf 1\otimes\sigma$ the element of $S_{n+1}$
for which $1$ is a fixed point and which permutes the $n$ numbers  $2,3,\dots,n+1$ according to $\sigma$. Note that, for each $k\ge2$, the permutation $\pi_1\pi_2\dotsm\pi_{k-1}$
puts $k$ on the first place, leaving the order of the other elements unchanged.
 Hence,
\begin{align}
P_{n+1}&=\frac1{(n+1)!}\sum_{\pi\in S_{n+1}}\Psi_\pi=\frac1{(n+1)!}\sum_{k=1}^{n+1}\sum_{\substack{\pi\in S_{n+1}\\ \pi(k)=1}}\Psi_\pi\notag\\
&=\frac1{(n+1)!}\sum_{k=1}^{n+1}\sum_{\sigma\in S_n}\Psi_{(\mathbf 1\otimes\sigma)\pi_1\pi_2\dotsm\pi_{k-1}}\notag\\
&=\frac1{n+1}\sum_{k=1}^{n+1}\frac1{n!}\sum_{\sigma\in S_n}(\mathbf1\otimes \Psi_\sigma)\Psi_1\Psi_2\dotsm\Psi_{k-1}\notag\\
&=\frac1{n+1}\sum_{k=1}^{n+1}(\mathbf 1\otimes P_n )\Psi_1\Psi_2\dotsm\Psi_{k-1}.\label{gdtrs5}
\end{align}
From here formula \eqref{jhfydfujuf} follows by taking the adjoint operators.
Formula \eqref{dstyufut} follows directly from \eqref{jhfydfujuf} if we mention that, for each $k=1,\dots,n$,
\begin{multline}\big(\Psi_k\Psi_{k-1}\dotsm\Psi_1 (h\otimes f^{(n)})\big)(t_1,\dots,t_{n+1})\\
=Q(t_1,t_{k+1})Q(t_2,t_{k+1})\dotsm Q(t_k,t_{k+1})h(t_{k+1})f(t_1,\dots\check t_{k+1},\dots,t_{n+1}), \label{fd5uei} \end{multline}
which can be easily checked by induction.
\end{proof}

In the definition \eqref{Q-symmetrization}  of the $Q$-symmetrization, $P_n$, was given through a rather abstract representation of $\pi$ as in \eqref{representation of permutation}. We will now derive an explicit formula for the action of  $P_n$.

\begin{proposition}[$Q$-symmetrization formula] For each $f^{(n)}\in \hc^{\otimes n}$, $n\ge2$, we have
\begin{equation}\label{jhcdthr}
(P_nf^{(n)})(t_1,\dots,t_n)=\frac 1{n!}\sum_{\pi\in S_n}Q_{\pi}(t_1,\dots,t_n) f^{(n)}(t_{\pi^{-1}(1)},\dots,t_{\pi^{-1}(n)}),
\end{equation}
where for $\pi\in S_n$
\begin{equation}\label{Q_pi_wzor}
Q_{\pi}(t_1, \ldots , t_n):=\!\!\! \prod_{\substack{ 1\le i < j \le n\\[1mm]
\pi(i) > \pi(j)}}
\!\!\! Q(t_i, t_j).
\end{equation}
In particular, for any  $f_1,\dots,f_n\in \hc$, we have:
\begin{equation}\label{kfdte6e}
(f_1\circledast\dots \circledast f_n)(t_1, \dots , t_n)
=\frac{1}{n!}\sum_{\pi\in S_n} Q_{\pi}(t_1, \dots ,
t_n)(f_{\pi(1)}\otimes\dots \otimes f_{\pi(n)})(t_1, \dots , t_n).
\end{equation}
\label{cftrst}\end{proposition}

\begin{proof}  It suffices to prove that, for each $\pi\in S_n$,
\begin{equation}\label{sths}
(\Psi_{\pi}f^{(n)})(t_1,\dots,t_n)=Q_{\pi^{-1}}(t_1,\dots,t_n)f^{(n)}(t_{\pi(1)},\dots,t_{\pi(n)}).
\end{equation}
A permutation $\pi\in S_n$ can be represented (not in a unique way, in general) as a reduced product of minimal number of transpositions, i.e., in the form \eqref{representation of permutation} in which $k$ is minimal possible. The number $k$ is called the length of $\pi$, and we will denote it by $|\pi|$. It is well known that $|\pi|$ is equal to the number of inversions of $\pi$, i.e., the number of $1\le i< j\le n$ such that $\pi(i)>\pi(j)$, see e.g., \cite{Knuth}

It follows from \eqref{operators Psi_j} that for an inversion  $\pi_j=(j,j+1)=\pi_j^{-1}$, formula \eqref{sths} trivially holds.
Hence, we can proceed by induction on the length
of $\pi=\pi_{j_1}\dotsm \pi_{j_k}$. If we
define $\zeta:=\pi_{j_1}\dotsm \pi_{j_{k-1}}$, so that
$\pi=\zeta\pi_{j_k}$, then $\zeta$ has length $k-1$, and using the
simplified notation $j:=j_k$,
 $\eta:=\zeta^{-1}$
and the induction assumption, we have:
\begin{align*}
&(\Psi_{\pi}f^{(n)})(t_1,\dots,t_n)=(\Psi_{\zeta}[\Psi_{\pi_{j}}
f^{(n)}])(t_1,\dots,t_n)\\
&\quad =
 Q_{\eta}(t_1, \dots , t_n)
(\Psi_{j}f^{(n)})(t_{\zeta(1)}, \dots , t_{\zeta(n)})\\
&\quad =Q_{\eta}(t_1, \dots , t_n)Q(t_{\zeta(j)},t_{\zeta(j+1)})
f^{(n)}(t_{\zeta(1)}, \dots ,
t_{\zeta(j-1)}, t_{\zeta(j+1)}, t_{\zeta(j)}, t_{\zeta(j+2)},
\dots , t_{\zeta(n)})\\
&\quad =Q_{\eta}(t_1, \dots , t_n)Q(t_{\zeta(j)},t_{\zeta(j+1)}) f^{(n)}(t_{\pi(1)}\dots, t_{\pi(n)}).
\end{align*}
Thus, we only need to prove that
\begin{equation}\label{Q_rho}
Q_{\rho}(t_1, \dots , t_n) = Q_{\eta}(t_1, \dots , t_n)Q(t_{\zeta(j)},t_{\zeta(j+1)}) .
\end{equation}
where $\rho:=\pi^{-1}=\pi_j\eta$.  Let  $1\le u < v \le n$. We have to consider the following cases.

\begin{itemize}
\item
If $\eta(u) \not\in \{j, j+1 \}$ and $\eta(v)\notin \{j, j+1 \}$, then
both $\eta(u)$ and $\eta(v)$ are fixed points for the transposition
$\pi_j$. Consequently,  $\rho(u)>\rho(v)$ if and only if
$\eta(u)>\eta(v)$. Thus,  the term $Q(t_u, t_v)$
appears in $Q_{\rho}(t_1, \dots , t_n)$ if and only if it appears in
$Q_{\eta}(t_1, \dots , t_n)$.

\item
If $\eta(u) \in \{j, j+1 \}$ and $\eta(v)\notin \{j, j+1 \}$, then
$\eta(v)$ is fixed by $\pi_j$, and, since $\pi_j(\eta(u))\in \{j,
j+1 \}$, the realtion between $\eta(u)$ and $\eta(v)$ is the same as
between $\rho(u)$ and $\rho(v)$.

\item
The case $\eta(u) \notin \{j, j+1 \}$ and $\eta(v)\in \{j, j+1 \}$
is analogous to the previous one.

\item
If $\eta(u)=j+1$ and $\eta(v)=j$, then $\rho(u)=j$ and
$\rho(v)=j+1$. Hence, $\eta(u)>\eta(v)$
and $\rho(u)<\rho(v)$, so that $\eta$  changes the order of the
pair $\{u < v \}$ while $\rho$ does not. Therefore,
$\eta$ has more inversions  than $\rho$: $|\eta|>|\rho|$. But this contradicts the assumption that $\pi$ (and equivalently $\rho$)
is in the reduced form, so that, in particular, $|\rho|=|\eta|+1$. Thus, this case is impossible.

\item The remaining case is $\eta(u)=j$ and $\eta(v)=j+1$, or,
equivalently $\zeta(j)=u$ $\zeta(j+1)=v$. Then
$\rho(u)=j+1$ and $\rho(v)=j$. Hence, $\eta(u)<\eta(v)$ while $\rho(u)
> \rho(v)$. Thus, the term $Q(t_u, t_v) = Q(t_{\zeta(j)}, t_{\zeta(j+1)})$
appears in $Q_{\rho}(t_1, \dots , t_n)$ but not in $Q_{\eta}(t_1,
\dots , t_n)$.
\end{itemize}
Thus, \eqref{Q_rho} is proven.
\end{proof}

We finish this section with the remarkable anyon exclusion principle, which was shown by Goldin and Majid \cite{GM}.

\begin{proposition}[\cite{GM}]\label{lfuig}
Assume that the function $Q$ is given by \eqref{kdsty} in which $q\ne1$ is an $N$-th root of one, i.e., $q^N=1$, for some $N\ge2$. Then, for each $f\in\hc$, $f^{\circledast N}=0$.
\end{proposition}

\begin{proof}
Since the proof of this statement is rather short, we present it here. For each $n\in\mathbb N$, define the $q$-number $$[n]_q:=1+q+q^2+\dots+q^{n-1}=(1-q^n)/(1-q),$$
and the $q$-factorial $[n]_q!:=[n]_q[n-1]_q\dotsm 1$.
 We state that, for any $(t_1,\dots,t_n)\in T^{(n)}$ with $t_1<t_2,\dots<t_n$, we have
 \begin{equation}\label{iue6w5}  f^{\circledast n}(t_1,\dots,t_n)=\frac{[n]_q!}{n!}\, f^{\otimes n}(t_1,\dots,t_n).\end{equation}
 This can be easily checked by induction in $n$ through formula \eqref{dstyufut}. (Note that, when applying formula \eqref{dstyufut}, we still have,   $t_1<t_2<\dots<t_{k-1}<t_{k+1}<\dots<t_n$ for each $k=1,\dots,n$.) By substituting $n=N$ into \eqref{iue6w5} and noting that $[N]_q=0$, we get the statement.
\end{proof}

\begin{remark}\label{vfydyr}
Note that, in the fermi case, for any $f,g\in \hc$, we have $f\wedge g\wedge f=0$ where $\wedge$ denotes antisymmetric tensor product. However, an analogous statement fails in the general case for anyons.  For example, for $q^3=1$, the function $f\circledast g\circledast f^{\circledast2}$ is generally speaking not equal to zero, even though  $g\circledast f^{\circledast3}=0$.
\end{remark}

\section{$Q$-Fock space and fundamental operators in it}\label{utyr65e}
We define a {\it $Q$-Fock space\/} by
$$\mathcal F^Q(\mathcal H):=\bigoplus_{n=0}^\infty \hc^{\cd n}n!\, .$$
Thus, $\mathcal F^Q(\mathcal H)$ is the Hilbert space which consists of all sequences $F=(f^{(0)}, f^{(1)},f^{(2)},\dots)$ with $f^{(n)}\in  \hc^{\cd n}$ ($\hc^{\cd 0}:=\mathbb C$) satisfying
$$\|F\|^2_{\mathcal F^Q(\mathcal H)}:=\sum_{n=0}^\infty\|f^{(n)}\|^2_{\hc^{\cd n}}n!<\infty.$$
(The inner product in $\mathcal F^Q(\mathcal H)$ is  induced by the norm in this space.)
The vector $\Omega:=(1,0,0,\dots)\in \mathcal F^Q(\mathcal H)$ is called the {\it vacuum}.
We  denote by $\mathcal F_{\mathrm{fin}}^Q(\mathcal H)$ the subset of $\mathcal F^Q(\mathcal H)$ consisting of all finite sequences $$F=(f^{(0)},f^{(1)},\dots,f^{(n)},0,0,\dots)$$ in which $f^{(i)}\in \hc^{\cd i}$ for $i=0,1,\dots,n$, $n\in\mathbb N$. This space can be endowed with the topology of the topological direct sum of the $\hc^{\cd n}$ spaces. Thus, the convergence in $\mathcal F_{\mathrm{fin}}^Q(\mathcal H)$ means uniform finiteness of non-zero components and coordinate-wise convergence in $\hc^{\cd n}$.

For each $h\in\mathcal H_{\mathbb C}$, we define a {\it creation operator\/} $a^+(h)$ and an {\it annihilation operator\/} $a^-(h)$ as linear operators acting on $\mathcal F_{\mathrm{fin}}^Q(\mathcal H)$ given by
\begin{equation}\label{Q-creation/annihilation}
a^{+}(h)f^{(n)} := h \cd f^{(n)},\quad f^{(n)}\in \hc^{\cd n},\quad a^{-}(h):= (a^{+}(h))^*\restriction _{\mathcal F_{\mathrm{fin}}^Q(\mathcal H)}.
\end{equation}
Clearly, $a^+(h)$ acts continuously on $\mathcal F_{\mathrm{fin}}^Q(\mathcal H)$,  hence so does $a^-(h)$.

Note that the action of the creation operator is explicitly given through the right hand side of formula \eqref{dstyufut}. The following proposition gives an explicit formula for the action of the annihilation operator.

\begin{proposition}\label{kds5u6r}
For $h\in \hc$ and $f^{(n)}\in \hc^{\cd n}$, we
have:
 \[
(a^-(h)f^{(n)})(t_1, \dots , t_{n-1}) =
 n\int_{T}\overline{h(s)}\,f^{(n)}(s, t_1, \dots ,
t_{n-1})\,\sigma(ds).
\]
\end{proposition}

\begin{proof} Let $ \mathcal F(\mathcal H):= \bigoplus_{n=0}^{\infty} \hc^{\otimes n}n!$ be the weighted full Fock space over $\mathcal H$ with weights $n!$, and let $\mathcal F_{\mathrm{fin}}(\mathcal H)$ be the subspace of finite sequences in $\mathcal F(\mathcal H)$. Free creation and annihilation operators are defined on $ \mathcal F_{\mathrm{fin}}(\mathcal H)$ by the formulas
$$a_{\mathrm{free}}^+(h)f^{(n)}:=h\otimes f^{(n)},\quad (a_{\mathrm{free}}^-(h)f^{(n)})(t_1,\dots,t_{n-1}):=\int_T\overline{h(s)}f^{(n)}(s,t_1,\dots,t_{n-1})\,\sigma(ds)$$
 for $h\in\hc$ and $f^{(n)}\in\hc^{\otimes n}$,
 and clearly
 $a_{\mathrm{free}}^{-}(h)= (a_{\mathrm{free}}^{+}(h))^*\restriction _{\mathcal F_{\mathrm{fin}}(\mathcal H)}$.
Let $P: \mathcal F(\mathcal H)\to \mathcal F^Q(\mathcal H)$ be the orthogonal projection of $\mathcal F(\mathcal H)$ onto $\mathcal F^Q(\mathcal H)$. We note that
$ P \mathcal F_{\mathrm{fin}}(\mathcal H)=\mathcal F^Q_{\mathrm{fin}}(\mathcal H).$
We have
$$ a^+(h)= Pa_{\mathrm{free}}^+(h)P,$$
hence
 $$ a^-(h)= Pa_{\mathrm{free}}^-(h)P.$$
 Thus, for each $F\in \mathcal F^Q_{\mathrm{fin}}(\mathcal H)$
 $$ a^-(h)F=P\,a_{\mathrm{free}}^-(h)PF=P\,a_{\mathrm{free}}^-(h)F=a_{\mathrm{free}}^-(h)F,$$
 where we used that $a_{\mathrm{free}}^-(h)F\in \mathcal F^Q_{\mathrm{fin}}(\mathcal H)$, see Proposition~\ref{lhgufr7}.
\end{proof}

The following proposition gives a formula for the action of the annihilation operator on a $Q$-symmetric tensor product of vectors from $\hc$.

\begin{proposition}\label{ctst}
For any $h,f_1,f_2,\dots,f_{n}\in\hc$, we have
\begin{multline*}
a^-(h)f_1\cd f_2\cd\dotsm\cd f_{n} \\
=\int_T\overline{h(s)}\,\bigg[\sum_{k=1}^n f_k(s) \big(Q(s,\cdot)f_1\big)\cd\dotsm\cd \big(Q(s,\cdot)f_{k-1}\big)\cd  f_{k+1}\cd \dotsm \cd f_n\bigg]\sigma(ds).
\end{multline*}
\end{proposition}

\begin{proof} By \eqref{gdtrs5}
\begin{multline} f_1\cd f_2\cd\dotsm\cd f_n=P_n(f_1\otimes f_2\otimes\dots\otimes f_n)\\
=\frac1{n}\,
(\mathbf 1\otimes P_{n-1})\bigg[ \mathbf 1+\sum_{k=2}^n\Psi_1\Psi_2\dotsm\Psi_{k-1}\bigg](f_1\otimes f_2\otimes\dots\otimes f_n).\label{hdr6teu5}\end{multline}
Analogously to \eqref{fd5uei}, we conclude that
\begin{multline*} \Psi_1\Psi_2\dotsm\Psi_{k-1}(f_1\otimes f_2\otimes\dots\otimes f_n)(s,t_1,\dots,t_{n-1})\\
=Q(s,t_1)Q(s,t_2)\dotsm Q(s,t_{k-1}) (f_k\otimes f_1\otimes f_2\otimes\dots \otimes f_{k-1}\otimes f_{k+1}\otimes\dots\otimes f_n)(s,t_1,\dots,t_{n-1})\\
=f_k(s) \big(Q(s,t_1)f_1(t_1)\big)\big(Q(s,t_2)f_2(t_2)\big)\dotsm \big(Q(s,t_{k-1})f_{k-1}(t_{k-1})\big)f_{k+1}(t_{k})\dotsm f_n(t_{n-1}).
 \end{multline*}
Hence
\begin{multline}
(\mathbf 1\otimes P_{n-1})\Psi_1\Psi_2\dotsm\Psi_{k-1}(f_1\otimes f_2\otimes\dots\otimes f_n)(s,t_1,\dots,t_{n-1})\\
=f_k(s) \big(\big(Q(s,\cdot)f_1\big)\cd\dotsm\cd \big(Q(s,\cdot)f_{k-1}\big)\cd  f_{k+1}\cd \dotsm \cd f_n\big)(t_1,\dots,t_{n-1}).\label{fdr6te5ut}
\end{multline}
By \eqref{hdr6teu5}, \eqref{fdr6te5ut} and Proposition~\ref{kds5u6r}, the statement follows.
\end{proof}

It is well known that, in the fermion case, the creation and annihilation operators are bounded in the antisymmetric Fock space, and the norm of each $a^+(h)$ and $a^-(h)$ is $\|h\|_{\hc}$. So the natural question arises as to whether this property remains for other generalized statistics. The following proposition was proven by Liguori and Mintchev \cite{LM}.

\begin{proposition}[\cite{LM}]\label{jhfdytd}
For each $h\in \hc$, the operator $a^{+}(h)$ \rom(and so $a^-(h)$\rom) is
bounded on $\mathcal F^Q(\mathcal H)$ with norm $\le\|h \|_{\hc}$ if and
only if the kernel $Q$ is negative semidefinite, i.e.,
\begin{equation}\label{negative definite}
\int_{T^2}Q(s, t)f(s)\overline{f(t)} \,\sigma(ds)\,\sigma(dt)\leqslant 0
\end{equation}
for any $f\in B_0(T\mapsto \mathbb C)$, a complex-valued bounded measurable function $f$ on $T$
with compact support.
\end{proposition}

We will now show that, for each  anyon statistics with $q\ne-1$, the function $Q$ does not satisfy the condition of the above proposition.

\begin{proposition}\label{hdytd}
Assume that $Q(s, t)$ is an anyonic kernel \rom(so that  $Q(s, t)=q$ for $s<t$ with $q\in \mathbb C$,  $|q|=1$\rom).
Moreover, assume that there exist disjoint sets $\Delta_1,
\Delta_2 \in\mathcal B_0(T)$ such that $\sigma(\Delta_1)>0$,
$\sigma(\Delta_2)>0$ and for all $s\in \Delta_1$ and  $t\in
\Delta_2$ we have $s<t$. Then  condition $(\ref{negative
definite})$ is satisfied if and only if $q=-1$.
\end{proposition}

\begin{remark}
Evidently, in the above proposition, the additional assumption on the space $T$  is satisfied in any reasonable example.
\end{remark}

\begin{proof}
Clearly, for $q=-1$,  condition $(\ref{negative
definite})$ is satisfied. To show the opposite, we set
$a:=\sigma(\Delta_1)$, $b:=\sigma(\Delta_2)$ and
$g(t):=\frac{b}{a}\, \overline{q} \, \chi_{\Delta_1}(t) +
\chi_{\Delta_2}(t)$. Here $\chi_{\Delta}$ denotes the indicator
function of a set $\Delta$. Then
\begin{align*}
&\int_{T^2} Q(s, t) g(s)\,\overline{g(t)}\, \sigma(ds)\,\sigma(dt)\\
&\quad =
\int_{\{s<t\}} q g(s)\,\overline{g(t)} \,\sigma(ds)\,\sigma(dt)
+
\int_{\{s>t\}} \overline{q}\, g(s)\overline{g(t)} \,\sigma(ds)\,\sigma(dt)\\
&\quad = \int_{\{s<t\}} q g(s)\,\overline{g(t)}\, \sigma(ds)\,\sigma(dt)+
\int_{\{s'<t'\}} \overline{q}\,
g(t')\,\overline{g(s')}\, \sigma(ds')\,\sigma(dt')\\
&\quad = 2 \operatorname{Re} \left(q  \int_{\{s<t\}}  g(s)\,\overline{g(t)}\, \sigma(ds)\,\sigma(dt) \right)\\
&\quad = 2 \operatorname{Re} \left(q \int_{(\Delta_1\times \Delta_1)\cap
\{s<t\}}\frac{b^2}{a^2} \,\sigma(ds)\,\sigma(dt) + q
\int_{(\Delta_2\times \Delta_2) \cap \{s<t\}}
 \sigma(ds)\,\sigma(dt)\right.\\
 &\qquad \quad\left.+q
\int_{\Delta_1}\sigma(ds)\int_{\Delta_2}\sigma(dt)\,
\frac{b}{a}\,  \overline{q}\right)\\
&\quad =\operatorname{Re}(qb^2+qb^2+2b^2) = 2b^2(\operatorname{Re}(q)+1),
\end{align*}
 which is $\le0$ if and only if $q=-1$.
\end{proof}

\begin{remark} Note that the assumption of Proposition~\ref{jhfdytd} is  stronger than the assumption  of boundedness of $a^+(h)$. So Proposition~\ref{hdytd} does not exclude the possibility of $a^+(h)$ being bounded with norm $>\|h\|_{\hc}$. Let us  make the following observation.  Let $\Delta\in\mathcal B_0(T)$. Let $\mathcal F^Q_\Delta$ denote the closed linear subspace of $\mathcal F^Q(\mathcal H)$ spanned by the vectors $\Omega$ and $\chi_\Delta^{\cd n}$, $n\in\mathbb N$. Note that $\mathcal F^Q_\Delta$ is an infinite dimensional space if and only if $q^n\ne1$ for all $n\in\mathbb N$. Evidently, $\mathcal F^Q_\Delta$ is an invariant subspace under the action of the creation operator $a^+(\chi_\Delta)$.
Assume that $q\ne1$. Then, using \eqref{iue6w5}, we have, for each $n\in\mathbb N$,
\begin{align*}
\|\chi_\Delta^{\cd n}\|_{\mathcal F^Q(\mathcal H)}^2&=n!\, \|\chi_\Delta^{\cd n}\|^2_{\mathcal H_{\mathbb C}^{\cd n}}\\
&=n!\left|\frac{[n]_q!}{n!}\right|^2\,\sigma(\Delta)^n\\
&=\big|[n]_q!\big|^2\,\frac{\sigma(\Delta)^n}{n!}\,.
\end{align*}
Therefore, the norm of the   operator $a^+(\chi_\Delta)$ restricted to $\mathcal F^Q_\Delta$ is equal to
$$\sup_{n\in\mathbb N}\frac{\big|[n]_q\big|}{\sqrt n}\,
\sigma(\Delta)^{1/2}=\sup_{n\in\mathbb N}\frac{|1-q^n|}{|1-q|\,\sqrt n}\,\sigma(\Delta)^{1/2}\le \frac2{|1-q|}\,\sigma(\Delta)^{1/2}.$$
In the boson case ($q=1$), the operator $a^+(\chi_\Delta)$ restricted to $\mathcal F^Q_\Delta$ is unbounded.
\end{remark}

Our next aim is to discuss the creation and annihilation operators at points of the space $T$.
At least informally, for each $t\in T$ we may consider a delta function at $t$, denoted by $\delta_t$.
Then we can heuristically define $\partial_t^\dag:=a^+(\delta_t)$ and $\partial_t:=a^-(\delta_t)$, so that
$$\partial_t^\dag f^{(n)}=\delta_t\cd f^{(n)},\quad \partial_t f^{(n)}=nf^{(n)}(t,\cdot).$$
Thus,
$$
a^+(h)=\int_T\sigma(dt)\, h(t)\partial_t^\dag,\quad
a^-(h)=\int_T\sigma(dt)\, \overline{h(t)}\,\partial_t.
$$
Such integrals are, as usual, understood through
 the corresponding quadratic forms with test functions, e.g.\ \cite{RS2} (see also formulas \eqref{fydyrr}, \eqref{fdswrty} below).

\begin{remark} Note that, by Proposition~\ref{ctst}, for any $f_1,\dots,f_n\in\hc$, we have
$$
\partial_t f_1\cd f_2\cd\dotsm\cd f_{n}=
\sum_{k=1}^n f_k(t)\big(Q(t,\cdot)f_1\big)\cd\dotsm\cd \big(Q(t,\cdot)f_{k-1}\big)\cd  f_{k+1}\cd \dotsm \cd f_n. $$
\end{remark}

Let $B_0(T^n\mapsto \mathbb C)$ denote the space of all complex-valued bounded measurable functions on $T^n$ with compact support.
Fix any sequence of $+$ and $-$ of length $n\ge 2$, and denote
it by $(\sharp_1, \dots , \sharp_n)$. It is easy to see that, for any
$g^{(n)}\in B_0(T^n\mapsto \mathbb C)$,  the expression
\[
\int_{T^n}\sigma(dt_1)\dotsm \sigma(dt_n)\,g^{(n)}(t_1, \dots ,
t_n)\partial^{\sharp_1}_{t_1}\dotsm \partial^{\sharp_n}_{t_n},
\]
identifies a linear continuous operator on $\mathcal  F_{\mathrm{fin}}^Q(\mathcal H)$. Here
 we used the notation $\partial_t^+:=\partial_t^\dag$, $\partial^-_t:=\partial_t$. (In fact, the class of functions $g^{(n)}$ could be chosen significantly larger than $B_0(T^n\mapsto\mathbb C)$ but we are not going to discuss this.)

\begin{proposition}
The creation and annihilation operators satisfy the commutation
relations:
\begin{align}
\partial_s\partial^\dag_t &= \delta(s, t)+Q(s, t)\partial^\dag_t\partial_s,\label{fdytde} \\
\partial_s\partial_t &= Q(t,s)\partial_t\partial_s ,\label{jhgyufd}\\
\partial^\dag_s\partial^\dag_t &= Q(t, s)\partial^\dag_t\partial^\dag_s .\label{pr++}
\end{align}
Here  $\delta(s, t)$ is understood as:
\[
\int_{T^2}\sigma(ds)\sigma(dt)\,f^{(2)}(s, t)\delta(s, t) :=
\int_{T}\sigma(dt)\, f^{(2)}(t, t).
\]
Formulas \eqref{fdytde}--\eqref{pr++} make  rigorous sense after smearing with  (test) functions $g^{(2)}\in B_0(T^2\mapsto\mathbb C)$
 and using the corresponding quadratic forms.
\end{proposition}

\begin{proof} Analogously to the proof of Proposition~\ref{jvfytdy}, we conclude that $P_n=P_n\Psi_1$, from where \eqref{pr++} follows. Formula \eqref{jhgyufd} is then derived by taking the adjoint operators.

To show formula \eqref{fdytde}, we note that, by \eqref{dstyufut} and Proposition~\ref{kds5u6r},
\begin{multline*}
(a^-(g)a^+(h)f^{(n)})(t_1,\dots,t_n)=\int_T\sigma(ds)\,\overline{g(s)}\,\bigg[h(s)f^{(n)}(t_1,\dots,t_n)\\
\text{}+\sum_{k=1}^n Q(s,t_k)Q(t_1,t_k)\dotsm Q(t_{k-1},t_k)f^{(n)}(s,t_1,\dots,\check t_k,\dots,t_n)\bigg]
\end{multline*}
for any  $g,h\in B_0(T\mapsto \mathbb C)$, $f^{(n)}\in\hc^{\cd n}$.
On the other hand,
$$ \int_T\sigma(ds)\int_T\sigma(dt)\,\overline{g(s)}\,
h(t)Q(s,t)\partial_t^\dag\partial_s f^{(n)} =P_nu^{(n)},$$
where
$$ u^{(n)}(t_1,\dots,t_n):= n\int_T\sigma(ds)\,\overline{g(s)}\, (h(t_1)Q(s,t_1)) f^{(n)}(s,t_2,t_3,\dots,t_n).$$
From here, by \eqref{dstyufut}, the statement follows.
\end{proof}

We finish this section by introducing  neutral (or preservation) operators.
For a function $h\in L^{\infty}(T\mapsto \mathbb C, \sigma)$ we define a {\it neutral
operator\/} by
\begin{equation}\label{fydyrr}
a^0(h):= \int_{T} \sigma(ds) h(s)\partial^\dag_s\partial_s.
\end{equation}
A meaning to this formula is again given through the corresponding quadratic form: if
$f^{(n)},g^{(n)}\in\hc^{\cd n}$, then
\begin{align}
&(a^0(h)f^{(n)},g^{(n)})_{\mathcal F^Q(\mathcal H)}=\int_T\sigma(ds)h(s)
(\partial_s f^{(n)},\partial_s g^{(n)})_{\mathcal F^Q(\mathcal H)}\notag\\
&=(n-1)!\,n^2 \int_T\sigma(ds)h(s)\int_{T^{n-1}}\sigma(dt_1)\dotsm
\sigma(dt_{n-1}) f^{(n)}(s, t_1, \dots , t_{n-1})\,
\overline{g^{(n)}(s, t_1, \dots , t_{n-1})}\notag\\
&=(n)!\,n\int_{T^{n}}\sigma(dt_1)\dotsm
\sigma(dt_{n}) h(t_1)f^{(n)}(t_1, t_2, \dots , t_{n})\,
\overline{g^{(n)}(t_1, t_2, \dots , t_{n})}.\label{fdswrty}
\end{align}
We note that
\begin{align}
&\int_{T^{n}}\sigma(dt_1)\dotsm
\sigma(dt_{n}) h(t_1)f^{(n)}(t_1, t_2, \dots , t_{n})\,
\overline{g^{(n)}(t_1, t_2, \dots , t_{n})}\notag\\
&\quad= \int_{T^{n}}\sigma(dt_1)\dotsm
\sigma(dt_{n}) h(t_1)Q(t_1,t_2)f^{(n)}(t_2, t_1,t_3 \dots , t_{n})\,
\overline{Q(t_1,t_2)g^{(n)}(t_2, t_1,t_3 \dots , t_{n})}\notag\\
&\quad=\int_{T^{n}}\sigma(dt_1)\dotsm
\sigma(dt_{n}) h(t_2)f^{(n)}(t_1, t_2, \dots , t_{n})\,
\overline{g^{(n)}(t_1, t_2, \dots , t_{n})}\notag\\
&\quad=\dots=\int_{T^{n}}\sigma(dt_1)\dotsm
\sigma(dt_{n}) h(t_n)f^{(n)}(t_1, t_2, \dots , t_{n})\,
\overline{g^{(n)}(t_1, t_2, \dots , t_{n})}.\label{des}
\end{align}
Hence, by \eqref{fdswrty} and \eqref{des},
\begin{multline*}(a^0(h)f^{(n)},g^{(n)})_{\mathcal F^Q(\mathcal H)}
\\
=n!\int_{T^{n}}\sigma(dt_1)\dotsm \sigma(dt_{n})\big(h(t_1)+\dots+h(t_n)\big)f^{(n)}(t_1, t_2, \dots , t_{n})\,
\overline{g^{(n)}(t_1, t_2, \dots , t_{n})}.\end{multline*}
Since the function $h(t_1)+\dots+h(t_n)$ is symmetric in the classical sense and the function $f^{(n)}$
is $Q$-symmetric, the function $(h(t_1)+\dots+h(t_n))f^{(n)}(t_1,  \dots , t_{n})$ is $Q$-symmetric.
Hence, $a^0(h)$ is the continuous operator on  $\mathcal F_{\mathrm{fin}}^Q(\mathcal H)$ given by
\begin{equation}\label{dse5ws} (a^0(h)f^{(n)})(t_1,\dots,t_n)=\big(h(t_1)+\dots+h(t_n)\big)f^{(n)}(t_1,  \dots , t_{n})\end{equation}
for $f^{(n)}\in\mathcal H_{\mathbb C}^{\cd n}$. Note also that, for any $f_1,\dots,f_n\in\hc$,
$$
(a^0(h)f_1\cd\dotsm\cd f_n)(t_1,\dots,t_n)=P_n\big[(h(t_1)+\dots+h(t_n))f_1(t_1)\dotsm f_n(t_n)\big].
$$ Therefore,
\begin{equation}\label{dd} a^0(h)f_1\cd\dotsm \cd f_n=\sum_{i=1}^n f_1\cd\dotsm\cd f_{i-1}\cd(hf_i)\cd f_{i+1}\cd\dotsm\cd f_n.\end{equation}

It can be easily deduced from \eqref{dse5ws} that, if $h\ne0$, the operator $a^0(h)$ is always unbounded in $\mathcal F^Q(\mathcal H)$.

\begin{remark}Let $A$ be a bounded linear operator in $\hc$. In \cite{LM}, a {\it differential second quantization  of\/} $A$ was defined as a linear operator $d\Gamma(A)$ in $\mathcal F^Q(\mathcal H)$ with domain $\mathcal F_{\mathrm{fin}}^Q(\mathcal H)$
given by
\begin{equation}\label{vgdftu} d\Gamma(A)\restriction \hc^{\cd n}:=P_n(A\otimes \mathbf 1\otimes\dots\otimes \mathbf 1+
\mathbf 1\otimes A\otimes \mathbf 1\otimes\dots\otimes\mathbf 1+\mathbf 1\otimes\dots\otimes\mathbf1\otimes A).\end{equation}
We clearly have $a^0(h)=d\Gamma(M_h)$, where $M_h$ is the operator of multiplication by the function $h$. Note that, in this case, the operator
$$M_h\otimes \mathbf 1\otimes\dots\otimes \mathbf 1+
\mathbf 1\otimes M_h\otimes \mathbf 1\otimes\dots\otimes\mathbf 1+\mathbf 1\otimes\dots\otimes\mathbf1\otimes M_h$$
acts invariantly on $\hc^{\cd n}$, so that the $Q$-symmetrization operator, $P_n$, in formula \eqref{vgdftu} may now be omitted.

Note also that, in the case of $q$-commutation relations with $q$ being real and $-1<q<1$ (see \cite{BS}), a corresponding differential second quantization, introduced by Anshelevich in \cite{A1}, appears to  be always a bounded operator (Lemma~1 in \cite{A1}), whereas our neutral operators, $a^0(h)$, are unbounded.
\end{remark}

\section{Q-Hermite and $Q$-Charlier polynomials}\label{cfydry}

We will now introduce  noncommutative analogs of Gaussian and Poisson processes (white noise measures) for $Q$-commutation relations. We denote by $B_0(T)$ the set of all real-valued bounded Borel-measurable function on $T$ with compact support. Let $\lambda\in\mathbb R$ be a fixed parameter. We consider a family of operators $(\la f,\omega\ra)_{f\in B_0(T)}$ defined by
$$ \la f,\omega\ra=a^+(f)+\lambda a^0(f)+a^-(f).$$
Choosing $\lambda=0$ corresponds to the $Q$-Gaussian case, while $\lambda=1$ corresponds to the (centered) $Q$-Poisson. (We will actually refer to each case $\lambda\ne0$ as  $Q$-Poisson.)
Each operator $\la f,\omega\ra$ acts continuously on  $\mathcal F_{\mathrm{fin}}^Q(\mathcal H)$ and is Hermitian in $\mathcal F^Q(\h)$. In fact, it can be easily shown by analogy with the classical (boson) case, see e.g.\ \cite{RS2,L1}, that each  $F\in \mathcal F_{\mathrm{fin}}^Q(\mathcal H)$ is an analytic vector for each operator 
$\la f,\omega\ra$ with $f\in B_0(T)$. Hence, each 
  operator $\la f,\omega\ra$ is essentially self-adjoint on $\mathcal F_{\mathrm{fin}}^Q(\mathcal H)$ (compare with \cite[Proposition~3]{LM}).

If we denote
$$\omega(t):=\partial_t^\dag+\lambda\partial_t^\dag\partial_t+\partial_t,\quad t\in T,$$
then, using our usual notation,
$$ \la f,\omega\ra=\int_T\sigma(dt)f(t)\omega(t),\quad f\in B_0(T),$$ which justifies the notation $\la f,\omega\ra$.

Let $\mathscr P$ denote the complex unital $*$-algebra generated by $(\la f,\omega\ra)_{f\in B_0(T)}$, i.e., the algebra of noncommutative polynomials in the variables $\la f,\omega\ra$. In particular, elements of $\mathscr P$ are linear operators acting on $\mathcal F_{\mathrm{fin}}^Q(\mathcal H)$, and for each $p\in\mathscr P$, $p^*$ is the adjoint operator of $p$ in $\mathcal F^Q(\h)$.
We define a vacuum state on $\mathscr P$ by
\begin{equation}\label{ty6w65uw6}\tau(p):=(p\Omega,\Omega)_{\fh},\quad p\in\mathscr P.\end{equation} We introduce a scalar product on $\mathscr P$
by
$$(p_1,p_2)_{L^2(\tau)}:=\tau(p_2^*p_1), \quad p_1,p_2\in\mathscr P.$$
Let $\mathscr P_0:=\{p\in\mathscr P\mid (p,p)_{L^2(\tau)}=0\}$, and define the noncommutative $L^2$-space $L^2(\tau)$ as the completion  of the quotient space $\mathscr P/\mathscr P_0$ with respect to the norm generated by the scalar product $(\cdot,\cdot)_{L^2(\tau)}$. Elements $p\in\mathscr P$ are treated as representatives of the equivalence classes from
$\mathscr P/\mathscr P_0$, and so $\mathscr P$ becomes a dense subspace of $L^2(\tau)$. (This has just been the  Gelfand--Naimark--Segal construction for $\mathscr P$ at the vacuum state $\tau$.)

\begin{proposition}\label{drawjh} \rom{(i)} The vacuum vector $\Omega$ is cyclic for the family of operators\linebreak $(\la f,\omega\ra)_{f\in B_0(T)}$.

\rom{(ii)} Consider a linear mapping $I:\mathscr P\to\fh$ defined by $Ip:=p\Omega$ for $p\in\mathscr P$. Then $Ip$ does not depend on the choice of $p\in\mathscr P/\mathscr P_0$ and $I $extends to a unitary operator $I:L^2(\tau)\to\fh$.
\end{proposition}

\begin{proof} Part i) can be easily shown by analogy with the boson case (see e.g.\ \cite{RS2} or \cite{L1}). Part ii) immediately follows from part i).
\end{proof}

For each $n\in\mathbb Z_+:=\{0,1,2,\dots\}$, we denote by $\mathscr P^{(n)}$ the subset of $\mathscr P$ which consists of all polynomials of degree $\le n$, i.e., the linear span of monomials
$$\la f_1,\omega\ra\dotsm\la f_k,\omega\ra=:\la f_1\otimes \dots\otimes f_k,\omega^{\otimes k}\ra$$
with $f_1,\dots,f_k\in B_0(T)$, $k\le n$, and constants. Let $\mathscr{MP}^{(n)}$ denote the closure of $\mathscr P^{(n)}$ in $L^2(\tau)$. ($\mathscr {MP}$ stands for {\it measurable polynomials}.) Let $\mathscr{OP}^{(n)}:=\mathscr{MP}^{(n)}\ominus\mathscr{MP}^{(n-1)}$, $n\in\mathbb N$, and $\mathscr{OP}^{(0)}:=\mathscr{MP}^{(0)}$, where the sign $\ominus$ denotes orthogonal difference in $L^2(\tau)$. ($\mathscr{OP}$ stands for {\it orthogonal polynomials}.) Since $\mathscr P$ is dense in $L^2(\tau)$ we get the orthogonal decomposition
$L^2(\tau)=\bigoplus_{n=0}^\infty \mathscr{OP}^{(n)}$.

For any $f_1,\dots,f_n\in B_0(T)$, the monomial $\la f_1\otimes \dots\otimes f_n,\omega^{\otimes n}\ra$ evidently belongs to $\mathscr {MP}^{(n)}$, and we denote its orthogonal projection onto $\mathscr{OP}^{(n)}$ by
  $\la f_1\otimes \dots\otimes f_n,{:}\,\omega^{\otimes n}\,{:}\ra$. The latter is a $Q$-analog of an (infinite-dimensional) Hermite polynomial if $\lambda=0$, respectively Charlier polynomial if $\lambda=1$.

  \begin{proposition}\label{ds5}
We have
$$\la f_1\otimes \dots\otimes f_n,{:}\,\omega^{\otimes n}\,{:}\ra=I^{-1}(f_1\cd\dotsm \cd f_n).$$
\end{proposition}

\begin{proof} By analogy with the proof of Proposition \ref{drawjh}, one sees that the set $I\mathscr{P}^{(n)}$ is dense in
$\bigoplus_{k=0}^n\hc^{\cd k}$. Therefore, $I\mathscr{MP}^{(n)}=\bigoplus_{k=0}^n\hc^{\cd k}$. Hence,
$I\mathscr{OP}^{(n)}=\hc^{\cd n}$. But the projection of the vector $\la f_1\otimes \dots\otimes f_n,\omega^{\otimes n}\ra\Omega$ onto $\hc^{\cd n}$ is $$a^+(f_1)\dotsm a^+(f_n)\Omega=f_1\cd\dotsm\cd f_n,$$ from where the statement follows.
\end{proof}

Let us consider the  topology on $B_0(T\mapsto\mathbb C)$ which yields the following notion of convergence: $f_n\to f$ as $n\to\infty$ means that there exists a set $\Delta\in\mathcal B_0(T)$ such that
$\operatorname{supp}(f_n)\subset\Delta$ for all $n\in\mathbb N$ and
$$\sup_{t\in T}|f_n(t)-f(t)|\to0\quad\text{as }n\to\infty.$$
By linearity and continuity we can extend the mapping
$$B_0(T)^n\ni (f_1,\dots,f_n)\mapsto\la f_1\otimes\dots\otimes f_n,\omega^{\otimes n}\ra\in \mathscr L(\mathcal F_{\mathrm{fin}}^Q(\h))$$
   to a mapping
$$ B_0(T^n\mapsto\mathbb C)\ni f^{(n)}\mapsto \la f^{(n)},\omega^{\otimes n}\ra \in \mathscr L(\mathcal F_{\mathrm{fin}}^Q(\h)).$$
Here $\mathscr L(\mathcal F_{\mathrm{fin}}^Q(\h))$ denotes the space of all linear continuous operators on $\mathcal F_{\mathrm{fin}}^Q(\h)$. We can also identify each $\la f^{(n)},\omega^{\otimes n}\ra$  with an element of $\mathscr{MP}^{(n)}$, and denote by $\la f^{(n)},{:}\,\omega^{\otimes n}\,{:}\ra$ the orthogonal projection of $\la f^{(n)},\omega^{\otimes n}\ra$ onto $\mathscr{OP}^{(n)}$. By Proposition~\ref{ds5},
$$\la f^{(n)},{:}\,\omega^{\otimes n}\,{:}\ra=\la P_nf^{(n)},{:}\,\omega^{\otimes n}\,{:}\ra  =I^{-1}P_nf^{(n)} .$$
We will also use the notation
$$\la f^{(n)},{:}\,\omega^{\otimes n}\,{:}\ra=:\int_{T^n}\sigma(dt_1)\dotsm\sigma(dt_n)\,f^{(n)}(t_1,\dots,t_n)\,{:}\,\omega(t_1)\dotsm\omega(t_n)\,{:}\,.$$

\begin{proposition}\label{fd5w}
We have the following recurrence relations: $\wick{\omega(t)}=\omega(t)$ and for $n\ge2$
\begin{align}
&\wick{\omega(t_1)\omega(t_2)\dotsm\omega(t_{n})}=\omega(t_1)\,\wick{\omega(t_2)\dotsm\omega(t_{n})}
-\lambda\sum_{i=2}^{n}\delta(t_1,t_i)\wick{\omega(t_2)\dotsm\omega(t_{n})}\notag\\
&\quad-
\sum_{i=2}^{n}\delta(t_1,t_i)Q(t_1,t_2)Q(t_1,t_3)\dotsm Q(t_1,t_{i-1})\wick{\omega(t_2)\dotsm\check\omega(t_i)\dotsm\omega(t_{n})}\, ,\label{fdtyde6e}
\end{align}
where $Q(t_1,t_1):=1$. Equality \eqref{fdtyde6e} is rigorously understood after smearing with test functions.
\end{proposition}

\begin{proof}
Since $\la f,\omega\ra\Omega=f$, we clearly have $\la f,\omega\ra=\la f,\wick{\omega}\ra$. Thus, we have to prove that, for each $n\ge 2$ and any
$f_1,\dots,f_n\in B_0(T)$,
\begin{align}
&\la f_1\otimes\dots\otimes f_n,\,\wick{\omega^{\otimes n}}\ra=\la f_1,\omega\ra \la f_2\otimes\dots\otimes f_n,\,\wick{\omega^{\otimes(n-1)}}\ra\notag\\
&\quad-\lambda\sum_{i=2}^n \la f_2\otimes\dots\otimes (f_1f_i)\otimes\dots\otimes f_n,\,\wick{\omega^{\otimes(n-1)}}\ra
-\sum_{i=2}^n\la u^{(n-2)}_i,\,\wick{\omega^{\otimes(n-2)}}\ra,\label{gfte}
\end{align}
where
\begin{multline*}
u^{(n-2)}_i(t_2,\dots,\check t_i,\dots,t_n)
=\int_{T}\sigma(dt_1)f_1(t_1)f_i(t_1)Q(t_1,t_2)Q(t_1,t_3)\dotsm Q(t_1,t_{i-1})\\
\times
f_2(t_2)\dotsm f_{i-1}(t_{i-1})f_{i+1}(t_{i+1})\dotsm f_n(t_n).\end{multline*}
By applying the unitary operator $I$ to the left and right hand sides of \eqref{gfte}, we see that equality \eqref{gfte} is equivalent to
$$ f_1\cd\dotsm\cd f_n=\la f_1,\omega\ra f_2\cd\dotsm \cd f_n-\lambda\sum_{i=2}^n f_2\cd\dotsm \cd(f_1f_i)\cd\dotsm\cd f_n
-\sum_{i=2}^n u_i^{(n-2)}.
$$ But the latter equality holds by virtue of the definition of the operator $\la f_1,\omega\ra$, see, in particular,  formula \eqref{dd}.
\end{proof}

\begin{remark}It follows from Proposition~\ref{fd5w} that, even  for $f_1,\dots,f_n\in B_0(T)$, the orthogonal polynomial
$\la f_1\otimes\dots\otimes f_n,\wick{\omega^{\otimes n}}\ra$ does not belong to $\mathscr P$, rather it is a polynomial of the form $\la f_1\otimes \dots\otimes f_n,\omega^{\otimes n}\ra+
\sum_{i=0}^{n-1}\la g^{(i)},\omega^{\otimes i}\ra$ with $g^{(i)}\in B_0(T^i\mapsto\mathbb C)$.
\end{remark}

Since $\omega(t)$ is represented through $\partial_t^\dag$ and $\partial_t$, it is natural to introduce a {\it $Q$-Wick ordering\/}: each product $\partial_s\partial_t^\dag$ must be replaced with $Q(s,t)\partial_t^\dag\partial_s$, until in each product of creation and annihilation operators, all  creation operators are to the left of all  annihilation operators. We will denote Wick ordering by $\Wick{\cdot}$.
In the boson case, it is well known that
\begin{equation}\label{dtrswe}\wick{\omega(t_1)\dotsm\omega(t_n)}=\Wick{\omega(t_1)\dotsm\omega(t_n)},\end{equation}
see e.g.\ \cite{IK}. So, it is important to know  whether this formula remains true for a general statistics.
In fact, a direct computation of the left and right hand sides of \eqref{dtrswe} for $n=3$ shows that the answer is always negative in the $Q$-Poisson case ($\lambda\ne0$) unless $Q\equiv 1$ (boson case), and is also negative in the $Q$-Gaussian case if $Q$ takes on non-real values (in particular, for anyons). The following result is worth comparing with \cite{BKS,Krolak}.

\begin{theorem}\label{bvfrs} If the function  $Q$ is real-valued, i.e., it takes values in $\{-1,1\}$, and if $\lambda=0$, i.e., $\omega(t)=\partial_t^\dag+\partial_t$
 \rom($Q$-Gaussian case\rom), then formula \eqref{dtrswe} holds.
\end{theorem}

\begin{proof} Denote by ${\cal P}^{(n)}(2)$ the collection of all ordered partitions $(I,J)$ of the set $\{1, \ldots , n \} $ into two
disjoint subsets, $I$ and $J$.  For each $(I,J)\in {\cal P}^{(n)}(2)$, we denote
$$ a_{(I,J)}(t_m):=\begin{cases}
\partial_{t_m}^\dag,&\text{if }m\in I,\\
\partial_{t_m},&\text{if }m\in J.
\end{cases}$$
Then
\begin{align}
\Wick{\omega(t_1) \dotsm \omega(t_n)}&=
\Wick{(\partial^\dag_{t_1}+\partial_{t_1})(\partial^\dag_{t_2}+\partial_{t_2})\dotsm (\partial^\dag_{t_n}+\partial_{t_n})}\notag\\
&=\sum_{(I,J)\in\mathcal P^{(n)}(2)}\Wick{a_{(I,J)}(t_1)
a_{(I,J)}(t_2)\dotsm a_{(I,J)}(t_n)}\,.\label{uiyr76re}
\end{align}
If $(I,J)\in\mathcal P^{(n)}(2)$, $I=\{i_1, \ldots , i_k \}$,
$J=\{j_{k+1}, \ldots , j_{n} \}$, then applying the $Q$-Wick ordering, we  get
\begin{equation}\label{uds57f}
\Wick{a_{(I,J)}(t_1)
a_{(I,J)}(t_2)\dotsm a_{(I,J)}(t_n)}=\partial^\dag_{t_{i_1}}\partial^\dag_{t_{i_2}}\dotsm
\partial^\dag_{t_{i_k}}\partial_{t_{j_{k+1}}}\dotsm \partial_{t_{j_{n}}} Q_{I, J}(t_1,
\ldots , t_n),
\end{equation}
where
\[
Q_{I, J}(t_1, \ldots , t_n) := \prod_{\substack {k\in
I,\, m\in J\\  m<k }}   Q(t_m, t_k).
\]
(In formula \eqref{uds57f}, we assume that $i_1<i_2<\dots<i_k$ and $j_{k+1}<j_{k+2}<\dots<j_n$.)
Thus, by \eqref{uiyr76re} and \eqref{uds57f}, we have:
\begin{equation}\label{gfydtrds}
\Wick{\omega(t_1) \dotsm \omega(t_n)}= \sum_{\substack{
I, J\in {\cal P}^{(n)}(2)\\  I=\{i_1, \ldots , i_k \} \\
J=\{j_{k+1}, \ldots , j_{n} \} }}
\partial^\dag_{t_{i_1}}\partial^\dag_{t_{i_2}}\dotsm
\partial^\dag_{t_{i_k}}\partial_{t_{j_{k+1}}}\dotsm \partial_{t_{j_{n}}} Q_{I, J}(t_1,
\ldots , t_n).
\end{equation}

We have
$$
\omega(s) \Wick{\omega(t_1)\dotsm \omega(t_n)} = \Wick{\partial^\dag_s \omega(t_1)\dotsm\omega(t_n)} +\partial_s
\Wick{\omega(t_1)\dotsm \omega(t_n)}.$$
If $I=\varnothing$, , i.e., $J=\{1, \ldots , n \}$, then there are no
creation operators in the corresponding term on the right hand side of  \eqref{gfydtrds}. Hence
\[
\partial_s\partial_{t_1}\dotsm \partial_{t_n} =\Wick{ \partial_s\partial_{t_1}\dotsm \partial_{t_n}}.
\]
If $I\ne \varnothing$, then, using \eqref{fdytde},
\begin{align}
&\partial_s\partial^\dag_{t_{i_1}}\partial^\dag_{t_{i_2}}\dotsm \partial^\dag_{t_{i_k}}
\partial_{t_{j_{k+1}}} \dotsm \partial_{t_{j_n}} Q_{I, J}(t_1,\dots,t_n) \notag\\
&\quad=\big[\delta(s,t_{i_1})\partial^\dag_{t_{i_2}}\dotsm \partial^\dag_{t_{i_k}}+Q(s,t_{i_1})\partial^\dag_{t_{i_1}}\partial_s
\partial^\dag_{t_{i_2}}\dotsm\partial^\dag_{t_{i_k}}\big]\partial_{t_{j_{k+1}}} \dotsm \partial_{t_{j_n}} Q_{I, J}(t_1,\dots,t_n) \notag\\
&\quad =\dots=\big[
\delta(s,t_{i_1})\partial^\dag_{t_{i_2}}\dotsm \partial^\dag_{t_{i_k}}+\delta(s,t_{i_2})Q(s,t_{i_1})\partial^\dag_{t_{i_1}}\partial^\dag_{t_{i_3}} \dotsm\partial^\dag_{t_{i_k}}\notag\\
&\qquad+\dots+\delta(s,t_{i_k})Q(s,t_{i_1})Q(s,t_{i_2})\dotsm Q(s,t_{i_{k-1}})\partial^\dag_{t_{i_1}}\dotsm \partial^\dag_{t_{i_{k-1}}}\notag\\
&\qquad
+Q(s,t_{i_1})Q(s,t_{i_2})\dotsm Q(s,t_{i_k})\partial^\dag_{t_{i_1}}\dotsm \partial^\dag_{t_{i_k}}\partial_s
\big]\partial_{t_{j_{k+1}}} \dotsm \partial_{t_{j_n}} Q_{I, J}(t_1,\dots,t_n)\notag\\
&\quad=\big[
\delta(s,t_{i_1})\partial^\dag_{t_{i_2}}\dotsm \partial^\dag_{t_{i_k}}+\delta(s,t_{i_2})Q(s,t_{i_1})\partial^\dag_{t_{i_1}}\partial^\dag_{t_{i_3}} \dotsm\partial^\dag_{t_{i_k}}\notag\\
&\qquad+\dots+\delta(s,t_{i_k})Q(s,t_{i_1})Q(s,t_{i_2})\dotsm Q(s,t_{i_{k-1}})\partial^\dag_{t_{i_1}}\dotsm \partial^\dag_{t_{i_{k-1}}}\big]\partial_{t_{j_{k+1}}} \dotsm \partial_{t_{j_n}} Q_{I, J}(t_1,\dots,t_n)\notag\\
&\qquad +\Wick{\partial_s a(t_1)\dotsm a(t_n)}.
\notag
\end{align}
Here and below, having fixed a partition $(I,J)$, we 
write $a(t_m)$ instead of  $a_{(I,J)}(t_m)$. 
We clearly have:
$$Q_{I, J}(t_1,\dots,t_n) = Q(t_1, t_{i_1})Q(t_2, t_{i_1})\dotsm Q(t_{i_1-1},
t_{i_1}) Q_{I\setminus \{i_1\}, J}(t_1, \ldots , t_{i_1-1},
t_{i_1+1}, \ldots , t_n).$$
Since the function $Q$ is real-valued, it is therefore symmetric. Hence,
\begin{align*}
&\delta(s,t_{i_1})\partial^\dag_{t_{i_2}}\dotsm \partial^\dag_{t_{i_k}}\partial_{t_{j_{k+1}}}\dotsm \partial_{t_{j_n}}Q_{I, J}(t_1,\dots,t_n)
 = \delta(s, t_{i_1}) Q(t_1, t_{i_1})\dotsm Q(t_{i_1-1}, t_{i_1})\\
 &\quad\times
\partial^\dag_{t_{i_2}}\dotsm \partial^\dag_{t_{i_{k}}}\partial_{t_{j_{k+1}}} \dotsm
\partial_{t_{j_n}} Q_{I\setminus \{i_1 \}, J}(t_1, \ldots , t_{i_1-1},
t_{i_1+1}, \ldots , t_n)\\
&\quad=\delta(s, t_{i_1})
Q(s,t_1)Q(s,t_2)\dotsm Q(s,t_{i_1-1}) \,\Wick{a(t_{1})\dotsm
 a(t_{i_{1}-1})a(t_{i_1+1})\dotsm a(t_n)}.
\end{align*}
Continuing by analogy, we therefore conclude that
\begin{align*}
&\partial_s \Wick{a(t_1)\dotsm a(t_n)}\\
&\quad=\sum_{l=1}^k\delta(s,t_{i_l})Q(s,t_1)Q(s,t_2)\dotsm Q(s,t_{i_l-1})\,\Wick{a(t_{1})\dotsm
 a(t_{i_{l}-1})a(t_{i_l+1})\dotsm a(t_n)}\\
 &\qquad+\Wick{\partial_s a(t_1)\dotsm a(t_n)}.
\end{align*}
Hence,
\begin{multline*} \omega(s)\Wick{\omega(t_1)\dotsm\omega(t_n)}\\=\Wick{\omega(s)\omega(t_1)\dotsm\omega(t_n)}+\sum_{l=1}^n\delta(s,t_l)\Wick
{\omega(t_1)\dotsm\omega(t_{l-1})\omega(t_{l+1})\dotsm\omega(t_n)},\end{multline*}
from where the statement follows.
\end{proof}

From now on, we will again treat the case of a general  function $Q$. Our next aim is to derive a representation of a monomial $\la f^{(n)},\omega^{\otimes n}\ra$ through orthogonal polynomials. We will first fix some notations.

Analogously to the symbol $\delta(s,t)$, we introduce a symbol $\delta(t_1,\dots,t_k)$ with $k\ge2$, which is understood as
$$\int_{T^k}\sigma(dt_1)\dotsm\sigma(dt_k)\,f^{(k)}(t_1,\dots,t_k)\delta(t_1,\dots,t_k):=\int_T\sigma(dt)\,f^{(k)}(t,\dots,t).$$
Let $\mathcal P^{(n)}_\pm$ denote the collection of all  partitions $\mathcal V$ of the set $\{1,2,\dots,n\}$ whose blocks are marked by $+1$ or $-1$ and such that, if a block has only one element (a singleton), then  the mark of this block is $+1$. For each marked partition $\mathcal V\in\mathcal P^{(n)}_\pm$, the expression $\wick{\omega(t_1)\dotsm\omega(t_n)}_{\mathcal V}$ will mean the following. Take $\wick{\omega(t_1)\dotsm\omega(t_n)}$\,. For each $B\in\mathcal V$ with mark $+1$ do the following: if $B$ is a singleton, then do nothing, and if $B=\{i_1,i_2,\dots,i_k\}$ with $k\ge2$ and $i_1<i_2<\dots<i_k$, then remove $\omega(t_{i_1}),\omega(t_{i_2}),\dots,\omega(t_{i_{k-1}})$ and multiply the result by $\lambda^{k-1}\delta(t_{i_1},t_{i_2},\dots,t_{i_{k}})$. For each $B=\{i_1,i_2,\dots,i_k\}\in\mathcal V$ with mark $-1$ (and hence $k\ge2$) do the following:
remove  $\omega(t_{i_1}),\omega(t_{i_2}),\dots,\omega(t_{i_{k}})$ and multiply the result by $\lambda^{k-2}\delta(t_{i_1},t_{i_2},\dots,t_{i_{k}})$. 

\begin{example} Consider the following marked partition of $\{1,2,\dots,6\}$:
\begin{equation}\label{hgyurf75er} \mathcal V=\{\big(\{1,6\},+1),\, (\{2,3,5\},-1),\, (\{4\},+1)\big\}.\end{equation}
Then
\begin{equation}\label{ftytde6e}\wick{\omega(t_1)\dotsm\omega(t_6)}_{\mathcal V}=\lambda^2\delta(t_1,t_6)\delta(t_2,t_3,t_5)\wick{\omega(t_4)\omega(t_6)}, \end{equation}
or in the smeared (integral) form
\begin{equation}\label{fty6w64u} \la f_1\otimes\dotsm\otimes f_6, \wick{\omega^{\otimes 6}}_{\mathcal V}\ra=\lambda^2\int_T (f_2f_3f_5)(t)\,\sigma(dt)\, \la f_4\otimes(f_1f_6),\wick{\omega^{\otimes 2}}\ra.
\end{equation}
\end{example}

We will also use the following notation: for ${\cal V}\in {\cal
P}^{(n)}_{\pm}$
\begin{equation}\label{hdyrjeu}
Q({\cal V}; t_1, \ldots , t_n) :=  \hspace{-1cm}
\prod_{\substack { B_1, B_2\in {\cal V}\\ m(B_1)=m(B_2)=-1,\\
\min B_1 < \min B_2 < \max B_1 < \max
B_2}}\hspace{-1.5cm} Q(t_{\min B_2}, t_{\max B_1}) \times
\hspace{-0.6cm}
\prod_{\substack { B_1, B_2\in {\cal V}\\ m(B_1)=+1,\, m(B_2)=-1,\\
\min B_2 < \max B_1 < \max B_2}} \hspace{-1cm} Q(t_{\min
B_2}, t_{\max B_1}).
\end{equation} Here, for a block $B$ from a marked partition  $\mathcal V\in\mathcal P^{(n)}_\pm$, $m(B)$ denotes the mark of $B$, while  $\min B$ ($\max B$, respectively) is the minimal (maximal, respectively) element of the block $B$.

\begin{theorem}[Wick rule for a product of fields]\label{eq243q}
 For each $n\in\mathbb N$, we have
\begin{equation}\label{hdtrsw}\omega(t_1)\dotsm\omega(t_n)=\sum_{\mathcal V\in\mathcal P^{(n)}_\pm}Q({\cal V}; t_1, \ldots , t_n)\wick{\omega(t_1)\dotsm\omega(t_n)}_{\mathcal V},\end{equation}
the formula making rigorous sense after smearing out with a function $f^{(n)}\in B_0(T^n\mapsto\mathbb C)$.
\end{theorem}

\setcounter{theorem}{5}
\begin{example}[continued] Let again a marked partition $\mathcal V\in \mathcal P^{(6)}_\pm$ be given by \eqref{hgyurf75er}. Then, by \eqref{hdyrjeu}, $Q(\mathcal V;t_1,\dots,t_6)=Q(t_2,t_4)$. Hence, by \eqref{ftytde6e},
$$ Q({\cal V}; t_1, \ldots , t_6)\wick{\omega(t_1)\dotsm\omega(t_6)}_{\mathcal V}= Q(t_2,t_4)\lambda^2\delta(t_1,t_6)\delta(t_2,t_3,t_5)\wick{\omega(t_4)\omega(t_6)}\,.$$
Fix any test functions $f_1,\dots,f_6$. Then, in the decomposition of $\la f_1,\omega\ra\dotsm\la f_6,\omega\ra$ according to the Wick rule, the term corresponding to the marked partition $\mathcal V$ has the form
\begin{equation}\label{iufd64w6u45}\lambda^2\Big\la f_4\otimes \Big(f_1f_6\cdot\int_T\sigma(dt)
(f_2f_3f_5)(t)Q(t,\cdot)\Big),\wick{\omega^{\otimes 2}}\Big\ra\end{equation}
(compare with \eqref{fty6w64u}, which is the special case of \eqref{iufd64w6u45} when $Q\equiv1$.)
Formula \eqref{iufd64w6u45} illustrates the difference between blocks having mark $+1$ and  blocks having mark $-1$. Indeed, in the marked partition  \eqref{hgyurf75er}, the block $\{2,3,5\}$ has mark $-1$, and so the function $(f_2f_3f_5)(t)$ times $Q(t,\cdot)$ is integrated against the measure $\sigma(dt)$. On the other hand, the blocks $\{4\}$ 
and $\{1,6\}$  have mark $=+1$, and so both functions $f_4$ and $f_1f_6$ appearing in \eqref{iufd64w6u45} are not integrated against $\sigma$.
\end{example}

\setcounter{theorem}{7}

\begin{proof}[Proof of Theorem~\ref{eq243q}] We prove formula \eqref{hdtrsw} by induction. It trivially holds for $n=1$. Assume that \eqref{hdtrsw} holds for $n$. Fix any $\mathcal V\in\mathcal P_\pm^{(n)}$, which we will treat as the corresponding collection of marked partitions of the set $\{2,3,\dots,n+1\}$. Denote by  $B_1,B_2,\dots,B_k$  the blocks of $\mathcal V$ which have mark $+1$. Let $i_j:=\max B_j$, $j=1,\dots,k$, and assume that $i_1<i_2<\dots<i_k$.
By Proposition~\ref{fd5w}, 
\begin{align*}
&\omega(t_1)\wick{\omega(t_{i_1})\dotsm\omega(t_{i_k})}=
\wick{\omega(t_1)\omega(t_{i_1})\dotsm\omega(t_{i_k})}+\sum_{j=1}^k\lambda\delta(t_1,t_{i_j})\wick{\omega(t_{i_1})\dotsm\omega(t_{i_k})}\\
&\quad+\sum_{j=1}^k\delta(t_1,t_{i_j})Q(t_1,t_{i_1})Q(t_1,t_{i_2})\dotsm Q(t_1,t_{i_{j-1}})\wick{\omega(t_{i_1})\dotsm
\check\omega(t_{i_j})\dotsm\omega(t_{i_k})}\,.
\end{align*}
Hence, 
\begin{align*}
&\omega(t_1)\,Q({\cal V}; t_2, \ldots , t_{n+1})\wick{\omega(t_2)\dotsm\omega(t_{n+1})}_{\mathcal V}\\
&\quad
=Q({\cal V}; t_2, \ldots , t_{n+1}) \bigg[\wick{\omega(t_1)\omega(t_2)\dotsm\omega(t_{n+1})}_{\mathcal V^{(1)}}
+\sum_{j=1}^{k}\lambda\delta(t_1,t_{i_j})\wick{\omega(t_2)\dotsm\omega(t_{n+1})}_{\mathcal V}\\
&\qquad + \sum_{j=1}^k \delta(t_1,t_{i_j}) Q(t_1,t_{i_1})Q(t_1,t_{i_2})\dotsm Q(t_1,i_{j-1})\big(
\wick{\omega(t_2)\dotsm\omega(t_{n+1})}_{\mathcal V}
\big)_{i_j}^\vee\bigg].
\end{align*}
Here $\mathcal V^{(1)}$ denotes the element of $\mathcal P^{(n+1)}_\pm$ which is obtained from $\mathcal V$ by adding the singleton $\{1\}$, marked $+1$, and $\big(
\wick{\omega(t_2)\dotsm\omega(t_{n+1})}_{\mathcal V}
\big)_{i_j}^\vee$ is obtained from $\wick{\omega(t_2)\dotsm\omega(t_{n+1})}_{\mathcal V}$ by removing $\omega(t_{i_j})$.
Therefore,
\begin{align*}
&\omega(t_1)\,Q({\cal V}; t_2, \ldots , t_{n+1})\wick{\omega(t_2)\dotsm\omega(t_{n+1})}_{\mathcal V}\\
&\quad
=Q(\mathcal V^{(1)}; t_1,\dots , t_{n+1}) \wick{\omega(t_1)\omega(t_2)\dotsm\omega(t_{n+1})}_{\mathcal V^{(1)}}\\
&\qquad
+\sum_{j=1}^k\sum_{l=2}^3 Q(\mathcal V^{(l)}_j; t_1,\dots , t_{n+1}) \wick{\omega(t_1)\omega(t_2)\dotsm\omega(t_{n+1})}_{\mathcal V^{(l)}_j},
\end{align*}
where $\mathcal V^{(l)}_j$ denotes the element of $\mathcal P^{(n+1)}_\pm$ which is obtained from $\mathcal V$ by adding
1 to the block containing $i_j$ and leaving the mark of this block to be $+1$ if $l=2$, respectively changing the mark of this block to $-1$ if $l=3$. From here formula \eqref{hdtrsw} for $n+1$ immediately follows.
\end{proof}

By applying the vacuum state $\tau$ to the left and right hand sides of \eqref{hdtrsw}, we get

\begin{corollary}[Moments formula]\label{fyde} For any $f^{(n)}\in B_0(T^n\mapsto\mathbb C)$, we
have
\begin{equation}\label{lgyuti}
\tau(\la f^{(n)},\omega^{\otimes n}\ra)=\sum_{\mathcal V\in \mathcal P^{(n)}_{\ge2}}\int_{T^n}f^{(n)}(t_1,\dots,t_n)
Q(\mathcal V;t_1,\dots,t_n)\prod_{B\in\mathcal V}\lambda^{|B|-2}\, \delta(dt_B).
\end{equation}
Here $\mathcal P_{\ge2}^{(n)}$ denotes the collection of all partitions $\mathcal V$ of $\{1,\dots,n\}$ such that each block
$B\in\mathcal V$ has at least two elements, i.e., $|B|\ge2$. For any subset $B=\{i_1,i_2,\dots,i_k\}$ of $\{1,\dots,n\}$
\rom($k\ge2$\rom), $\delta(dt_B):=\delta(dt_{i_1}\times dt_{i_2}\times\dots\times dt_{i_k})$, where
$$\int_{T^k} g^{(k)}(s_1,\dots,s_k)\,\delta(ds_1\times\dots\times ds_k):=\int_{T^k}g^{(k)}(s,s,\dots,s)\,\sigma(ds).$$
\rom(Note that $\delta(ds_1\times\dots\times ds_k)$ is a measure on $(T^k,\mathcal B(T^k))$.\rom)
Furthermore,
\begin{equation}\label{cdtdtrsd} Q(\mathcal V;t_1,\dots,t_n):=\hspace{-1cm} \prod_{\substack
{ B_1, B_2 \in {\cal V},\\ \min B_1 < \min B_2 <
\max B_1 < \max B_2 }}\hspace{-1cm} Q(t_{\min B_2},
t_{\max B_1}).\end{equation}
\end{corollary}

The reader is advised to compare the following corollary with \cite[Theorem~4.4]{BS_1994}, which deals with a Gaussian process for discrete commutation relations \eqref{iftru8rf}, and with \cite[Lemma~7.5]{A2}, which deals with a Poisson process for the $q$-deformed commutation relations \eqref{dtrs54t}.  Recall that we denoted by $\mathscr P$ the complex unital $*$-algebra generated by $(\la f,\omega\ra)_{f\in B_0(T)}$, and the state $\tau$ on $\mathscr P$ is given by \eqref{ty6w65uw6}.

\begin{corollary}
The state $\tau$ on $\mathscr P$ is tracial, i.e., it satisfies $\tau(p_1p_2)=\tau(p_2p_1)$ for all $p_1,p_2\in\mathscr P$, if and only if

\begin{itemize}
\item $Q\equiv 1$ and $\lambda\ne0$; or

\item the function $Q$ is real-valued, i.e., it takes values in $\{-1,1\}$, and $\lambda=0$.
\end{itemize}

\end{corollary}

\begin{proof} We first consider the Poisson case, i.e., $\lambda\ne0$. We take any disjoint sets $\Delta_1,\Delta_2\in\mathcal B_0(T)$ and set $f_i:=\chi_{\Delta_1}$, $i=1,3,5$, and $f_i:=\chi_{\Delta_2}$, $i=2,4$.  Using formula \eqref{lgyuti}, we get
$$\tau(\la f_1,\omega\ra\dotsm\la f_5,\omega\ra)=\lambda\sigma(\Delta_1)\sigma(\Delta_2),$$
while
$$\tau(\la f_5,\omega\ra\la f_1,\omega\ra\dotsm\la f_4,\omega\ra)=\lambda\int_{\Delta_1}\sigma(dt_1)\int_{\Delta_2}\sigma(dt_2)Q(t_2,t_1).$$
Hence, $\tau$ is not tracial if  $Q\not\equiv 1$. In the classical case, $Q\equiv1$, the state is trivially tracial, as the operators $(\la f,\omega\ra)_{f\in B_0(T)}$ commute.

Next, we consider the Gaussian case, $\lambda=0$. With the same functions $f_1,\dots,f_4$ as above, we get
$$\tau(\la f_1,\omega\ra\dotsm\la f_4,\omega\ra)=
\int_{\Delta_1}\sigma(dt_1)\int_{\Delta_2}\sigma(dt_2)Q(t_2,t_1),$$
while $$\tau(\la f_4,\omega\ra\la f_1,\omega\ra\dotsm\la f_3,\omega\ra)=\int_{\Delta_1}\sigma(dt_1)\int_{\Delta_2}\sigma(dt_2)Q(t_1,t_2).$$
Hence, for the state $\tau$ to be tracial, it is necessary that the function $Q$ be symmetric, i.e., it must take values in $\{-1,1\}$. Let us show that, in the latter case, the state $\tau$ is indeed tracial. 

For $\lambda=0$, formula \eqref{lgyuti} reduces to 
\begin{equation}\label{f76r7er} \tau(\la f^{(n)},\omega^{\otimes n}\ra)=\sum_{\mathcal V\in \mathcal P^{(n)}_{2}}\int_{T^n}f^{(n)}(t_1,\dots,t_n)
Q(\mathcal V;t_1,\dots,t_n)\prod_{B\in\mathcal V}\, \delta(dt_B),\end{equation}
where $\mathcal P^{(n)}_{2}$ denotes the collection of all partitions $\mathcal V$ of $\{1,\dots,n\}$ such that each block $B\in\mathcal V$ has exactly two elements. To prove that $\tau$ is tracial it suffices to show that, for any $f_1,\dots,f_{n+1}\in B_0(T)$, $n$ odd,
\begin{equation}\label{kjfr7}
\tau\big(\la f_1,\omega\ra\dotsm \la f_n,\omega\ra\la f_{n+1},\omega\ra\big)=\tau\big(\la f_{n+1},\omega\ra\la f_1,\omega\ra\dotsm \la f_n,\omega\ra\big).
\end{equation}

Let us fix any partition $\mathcal V\in\mathcal P_2^{(n+1)}$. Let $i\in\{1,\dots,n\}$ be such that $\{i,n+1\}$ is a block from $\mathcal V$.  By \eqref{cdtdtrsd},
\begin{equation}\label{kiyure67egg} Q(\mathcal V;t_1,\dots,t_{n+1})=\hspace{-1cm} \prod_{\substack
{ B_1, B_2 \in {\cal V},\\ \min B_1 < \min B_2 <
\max B_1 < \max B_2 <n+1}}\hspace{-1cm} Q(t_{\min B_2},
t_{\max B_1})\hspace{-5mm}\prod_{\substack{B\in\mathcal V\\ \min B<i<\max B}}Q(t_i,t_{\max B}).\end{equation}
Define a permutation $\pi\in S_{n+1}$ by $\pi(j):=j+1$, $j=1,\dots,n$, $\pi(n+1):=1$. Then the sets $\pi B$ with $B\in\mathcal V$ form a new partition from $\mathcal P_2^{(n+1)}$. We denote this partition by $\pi \mathcal V$. Note that $\{1,i+1\}$ is a block from $\pi\mathcal V$. Using that the function $Q$ is symmetric, we get, analogously to \eqref{kiyure67egg},
$$
Q(\pi\mathcal V;t_1,\dots,t_{n+1})=\hspace{-1cm} \prod_{\substack
{ B_1, B_2 \in \pi{\cal V},\\ 1<\min B_1 < \min B_2 <
\max B_1 < \max B_2 }}\hspace{-1cm} Q(t_{\min B_2},
t_{\max B_1})\hspace{-5mm}\prod_{\substack{B\in\pi\mathcal V\\ \min B<i+1<\max B}}Q(t_{i+1},t_{\min B}).
$$
Hence
\begin{equation}\label{ifre75f}
Q(\pi\mathcal V;t_{n+1},t_1,\dots,t_{n})=\hspace{-1cm} \prod_{\substack
{ B_1, B_2 \in {\cal V},\\ \min B_1 < \min B_2 <
\max B_1 < \max B_2 <n+1}}\hspace{-1cm} Q(t_{\min B_2},
t_{\max B_1})\hspace{-5mm}\prod_{\substack{B\in\mathcal V\\ \min B<i<\max B}}Q(t_i,t_{\min B}).
\end{equation}
By \eqref{kiyure67egg} and \eqref{ifre75f},
\begin{align}
&\int_{T^{n+1}}f_{n+1}(t_1)f_1(t_2)\dotsm f_n(t_{n+1})
Q(\pi\mathcal V;t_1,\dots,t_{n+1})\prod_{B\in\pi\mathcal V}
\delta(dt_B)\notag\\
&\quad= \int_{T^{n+1}}f_{n+1}(t_{n+1})f_1(t_1)\dotsm f_n(t_{n})
Q(\pi\mathcal V;t_{n+1},t_1,\dots,t_{n})\prod_{B\in\mathcal V}
\delta(dt_B)\notag\\
&\quad= \int_{T^{n+1}}f_1(t_1)\dotsm f_{n+1}(t_{n+1})
Q(\mathcal V;t_1,\dots,t_{n+1})\prod_{B\in\mathcal V}
\delta(dt_B),\label{g7e675r}
\end{align}
where we used that $t_{\min B}=t_{\max B}$ for 
$\delta(dt_{B})$-a.a.\ $(t_{\min B},t_{\max B})$. Formula \eqref{kjfr7} now follows from \eqref{f76r7er} and \eqref{g7e675r}.
\end{proof}

\section{$Q$-cumulants and $Q$-independence}\label{gtuyr6}
Our next aim is to introduce  $Q$-deformed cumulants. Let $\mathfrak F $ be a complex separable Hilbert space, and let $\mathfrak D$ be a linear subspace of $\mathfrak F$. Let $(\la f,\xi\ra)_{f\in B_0(T)}$ be a family linear symmetric operators acting on $\mathfrak D$, i.e., $\la f,\xi\ra:\mathfrak D\to\mathfrak D$, and such that the mapping $B_0(T)\ni f\mapsto \la f,\xi\ra$ is linear. We also assume that
 \begin{equation}\label{ftydr}\la f,\xi\ra=0\ \text{if and only if $f=0$ $\sigma$-a.e.}\end{equation}

\begin{remark} Analogously to Section~\ref{cfydry}, the reader may intuitively think of $\xi(t)$ as a field at point $t\in T$, while $\la f,\xi\ra=\int_T\sigma(dt)\,f(t)\xi(t)$. 
\end{remark}

For a fixed vector $\Psi\in\mathfrak D$ with $\|\Psi\|=1$, we define moments of  $(\la f,\xi\ra)_{f\in B_0(T)}$ by
$$\tau(\la f_1,\xi\ra\dotsm \la f_n,\xi\ra):=(\la f_1,\xi\ra\dotsm \la f_n,\xi\ra\Psi,\Psi)_{\mathfrak F},\quad f_1,\dots,f_n\in B_0(T).$$
Extending by linearity, we get a state (expectation) $\tau$ on the unital $*$-algebra generated by the operators $(\la f,\xi\ra)_{f\in B_0(T)}$. We will assume that, for each $n\in\mathbb N$, there exists a complex-valued, Radon measure $m_n$ on $T^n$ satisfying
\begin{equation}\label{kfytdy}
\tau(\la f_1,\xi\ra\dotsm \la f_n,\xi\ra)=\int_{T^n}f_1(t_1)\dotsm f_n(t_n)\,m_n(dt_1\times\dotsm\times dt_n), \quad f_1,\dots,f_n\in B_0(T).
\end{equation}
(Evidently each measure $m_n$ is uniquely defined.) Inspired by formula \eqref{lgyuti}, we now
give the following

\begin{definition} For each $n\in\mathbb N$, the {\it $n$-th $Q$-cumulant measure\/} of the operators (noncommutative random variables)
$(\la f,\xi\ra)_{f\in B_0(T)}$ is defined as the complex-valued Radon measure $c_n$ on $(T^n,\mathcal B(T^n))$ given  recursively through \begin{gather*}c_1(dt):=m_1(dt),\\
m_n(dt_1\times\dots\times dt_n)=\sum _{\mathcal V\in \mathcal P^{(n)}}Q(\mathcal V;t_1,\dots,t_n)\prod_{B\in\mathcal V}c_{|B|}(dt_B),\quad n\ge2.\end{gather*}
Here $\mathcal P^{(n)}$ denotes the collection of all partitions of $\{1,\dots,n\}$, the factor $Q(\mathcal V;t_1,\dots,t_n)$
is given by \eqref{cdtdtrsd}, and for each $B=\{i_1,i_2,\dots,i_k\}\in\mathcal V$,  $c_{|B|}(dt_B):=c_k(dt_{i_1}\times dt_{i_2}\times\dots\times dt_{i_k})$.
For any $f_1,\dots,f_n\in B_0(T)$, we define the {\it $n$-th $Q$-cumulant of\/}   $\la f_1,\xi\ra,\dots, \la f_n,\xi\ra$ by
\begin{equation}\label{cfydy} C_n(\la f_1,\xi\ra,\dots, \la f_n,\xi\ra):=\int_{T^n}f_1(t_1)\dotsm f_n(t_n)\,c_n(dt_1\times \dots\times dt_n).\end{equation}
\end{definition}

The following lemma shows the consistency of this definition.

\begin{lemma}
Let $f_1,\dots,f_n\in B_0(T)$ and let, for some $i\in\{1,\dots,n\}$, $f_i=0$ $\sigma$-a.e. Then $ C_n(\la f_1,\xi\ra,\dots, \la f_n,\xi\ra)=0$.
\end{lemma}

\begin{proof}
In view of formulas \eqref{ftydr} and  \eqref{kfytdy}, for any $g_1,\dots,g_k\in B_0(T)$, $k\in\mathbb N$, such that, for some $j\in\{1,\dots,k\}$, $g_j=0$ $\sigma$-a.e., we have
\begin{equation}\label{fytf}
\int_{T^k} g_1(t_1)\dotsm g_k(t_k)\,m_k(dt_1\times\dots\times dt_k)=0.
\end{equation}
It can be easily shown by induction that each cumulant measure $c_n$ is a finite sum of complex-valued measures
of the form
\begin{equation}\label{tsts}
R(t_1,\dots,t_n)\,m_{|B_1|}(dt_{B_1})\dotsm m_{|B_k|}(dt_{B_k}),
\end{equation}
where $\mathcal V=\{B_1,\dots,B_k\}\in\mathcal P^{(n)}$ and the function $R(t_1,\dots,t_n)$ is a finite product of functions $Q(t_u,t_v)$, were $u,v\in\{1,\dots,n\}$ belong to different blocks of the partition $\mathcal V$. Assume that the number $i\in\{1,\dots,n\}$ for which $f_i=0$ $\sigma$-a.e.\ belongs to $B_j\in\mathcal V$. Represent
$$ R(t_1, \dots , t_n)=R_1(t_1, \ldots , t_n)R_2(t_1, \ldots , t_n),$$
 where $R_1(t_1, \ldots , t_n)$ is a
product of $Q(t_{u}, t_{v})$ such that $u, v
\notin B_j$, and $R_2(t_1, \ldots , t_n)$ is a product of
$Q(t_{u}, t_{v})$ or $\overline{Q(t_{u}, t_{v})}$
such that $u \in B_j$ and $v \not\in B_j$. Then
\begin{align*}
&\int_{T^n}f_1(t_1)\dotsm f_n(t_n)R(t_1, \ldots ,
t_n)\,m_{|B_1|}(dt_{B_1})\dotsm m_{|B_k|}(dt_{B_k}) \\
&\quad= \int_{T^{n-|B_j|}}\bigotimes_{
l=1,\dots,k,\ l\ne j}
m_{|B_l|}(dt_{B_l})
\left(\prod_{u = 1,\dots,n,\ u\notin B_j}
f_{u}(t_{u})\right)
 R_1(t_1, \ldots ,
 t_n)\\
&\qquad\times
\int_{T^{|B_j|}}m_{|B_j|}(dt_{B_j})R_2(t_1, \ldots , t_n)
\prod_{v\in B_j}f_{v}(t_{v}),
\end{align*}
which is equal to $0$ by \eqref{fytf}.
\end{proof}

\begin{remark} We can heuristically think of a field $(\xi(t))_{t\in T}$, where $\xi(t):=\la\delta_t,\xi\ra$. Then, in view of formula \eqref{kfytdy},
$$\tau(\xi(t_1)\dotsm\xi(t_n))=m_n(dt_1\times\dots\times dt_n),$$
i.e., the measure $m_n$ gives the $n$-th moments of the filed  $(\xi(t))_{t\in T}$, while in view of formula \eqref{cfydy},  $c_n(dt_1\times\dots\times dt_n)$ is the $Q$-cumulant of $\xi(t_1),\dots,\xi(t_n)$:
$$ C_n(\xi(t_1),\dots,\xi(t_n))=c_n(dt_1\times\dots \times dt_n).$$
\end{remark}

Now that we have defined $Q$-cumulants, we can naturally introduce the notion of $Q$-independence.

\begin{definition} For $f_1,\dots,f_n\in B_0(T)$ ($n\ge2$), we will say that the operators (noncommutative random variables)
$\la f_1,\xi\ra,\dots,\la f_n,\xi\ra$ are {\it $Q$-independent\/} if, for any $k\ge2$ and any non-constant sequence $(j_1,j_2,\dots,j_k)$ of numbers from $\{1,\dots,n\}$,
$$C_k(\la f_{j_1},\xi\ra,\la f_{j_2},\xi\ra,\dots,\la f_{j_k},\xi\ra)=0.$$
\end{definition}

Let us consider the family of operators $(\la f,\omega\ra)_{f\in B_0(T)}$ as in Section~\ref{cfydry}. By Corollary~\ref{fyde}, $n$-th $Q$-cumulant measure of this family is given by
\begin{gather*}
c_1(dt_1)=0,\notag\\
 c_n(dt_1\times\dotsm\times dt_n)=\lambda^{n-2}
 \delta(dt_1\times\dots\times dt_n),\quad n\ge2,
 \end{gather*}
as we would expect for a Gaussian or a (centered) Poisson process, respectively.
Hence, for any $f_1,\dots,f_n\in B_0(T)$ ($n\ge2$) and any sequence $(j_1,j_2,\dots,j_k)$ of numbers from $\{1,\dots,n\}$, we have
$$ C_k(\la f_{j_1},\omega\ra,\dots,\la f_{j_k},\omega\ra)=\lambda^{k-2}\int_T f_{j_1}(t)\dotsm f_{j_k}(t)\,\sigma(dt). $$
Hence,
if  $f_if_j=0$ $\sigma$-a.e.\ for all $1\le i<j\le n$, the operators $\la f_1,\omega\ra,\dots,\la f_n,\omega\ra$ are $Q$-independent.

\section{$Q$-L\'evy processes}\label{kigf7ur7}

We are now in a position to introduce the notion of  $Q$-L\'evy processes.

\begin{definition}Let $(\la f,\xi\ra)_{f\in B_0(T)}$ be a family of operators as in Section~\ref{gtuyr6}. We call $(\la f,\xi\ra)_{f\in B_0(T)}$ a {\it $Q$-L\'evy process\/} if it satisfies the following conditions.

\begin{itemize}
\item[(i)] For any sets $\Delta_1,\dots,\Delta_n\in \mathcal B_0(T)$ which are mutually disjoint, the operators $\la \chi_{\Delta_1},\xi,\ra,\dots,\la \chi_{\Delta_n},\xi\ra$ are $Q$-independent (`independence of increments');

\item[(ii)] For any $\Delta_1,\Delta_2\in \mathcal B_0(T)$ such that $\sigma(\Delta_1)=\sigma(\Delta_2)$,
$$\tau(\la\chi_{\Delta_1},\xi\ra^n)=\tau(\la\chi_{\Delta_2},\xi\ra^n)\quad \text{for all }n\in\mathbb N.$$
 (`stationarity of increments').
\end{itemize}
\end{definition}

It is evident that, for each parameter $\lambda\in\mathbb R$, the operator field $(\la f,\omega\ra)_{f\in B_0(T)}$ from Section~\ref{cfydry} is a $Q$-L\'evy process. We will now discuss a rather general construction of  (a class of) $Q$-L\'evy processes, which is close in spirit  both to classical probability and to  free probability, and which includes the $Q$-Gaussian and $Q$-Poisson processes as  special cases.

Let $\nu$ be a probability measure on $\mathbb{R}$ and assume that
there exists $\varepsilon >0$ such that
\begin{equation}\label{dstrst}
\int_{\mathbb{R}}e^{{\varepsilon}|x|} \,\nu(dx) < \infty ,
\end{equation} or, equivalently, there exists $C>0$ such that, for all $n\in \mathbb{N}$,
\begin{equation}\label{cfdh}
\int_{\mathbb{R}}|x|^{n}\, \nu(dx) \le n!\,C^n.
\end{equation}
This assumption assures that the polynomials are dense in
$L^2(\mathbb{R}, \nu)$. We denote by $\mu_k$ the $k$-th order monomial on $\mathbb  R$, i.e.,
\begin{equation}\label{hcdydyd}\mathbb R\ni x\mapsto \mu_k(x):=x^k,\quad k\in\mathbb Z_+.\end{equation}
 In particular, $\mu_0\equiv1$.

Consider a function $Q:T^{(2)}\to S^1$ as above. We extend $Q$  by setting
$$Q(t_1,x_1,t_2,x_2):=Q(t_1,t_2),\quad (t_1,t_2)\in T^{(2)},\ (x_1,x_2)\in\mathbb R^2.$$
We  now set $$\mathcal G:=L^2(T\times \mathbb R,\sigma\otimes \nu)=\h\otimes L^2(\mathbb R,\nu),$$
 and construct the corresponding $Q$-symmetric Fock space $\mathcal F^Q(\mathcal G)$.
  For each $f\in B_0(T)$, we define an operator
$$\la f,\xi\ra:=a^{+}(f\otimes \mu_0)+a^0(f\otimes \mu_1)+a^-(f\otimes\mu_0)$$
on a proper domain $\mathfrak D$ in $\mathcal F^Q(\mathcal G)$. The domain $\mathfrak D$ consists of all finite sequences
$$F=(F^{(0)},F^{(1)},\dots,F^{(n)},0,0,\dots),\quad n\in\mathbb Z_+,$$
such that each $F^{(k)}$ with $k\ne0$ has the form
$$F^{(k)}(t_1,x_1,\dots,t_k,x_k)=P_k\left[\sum_{(i_1,i_2,\dots,i_k)\in\{0,1,\dots,N\}^k}f_{(i_1,i_2,\dots,i_k)}
(t_1,t_2,\dots,t_k)x_1^{i_1}x_2^{i_2}\dotsm x_{i_k}^{i_k}\right],$$
where $f_{(i_1,i_2,\dots,i_k)}\in \h_{\mathbb C}^{\otimes k}$ and $N\in\mathbb N$. Clearly, each operator $\la f,\xi\ra$ maps the domain $\mathfrak D$ into itself.

Note that, if the measure $\nu$ is concentrated at one point, $\lambda\in\mathbb R$, then
$\mathcal G=\h$ and
$(\la f,\xi\ra)_{f\in B_0(T)}$ is just the  $Q$-Gaussian/Poisson process  $(\la f,\omega\ra)_{f\in B_0(T)}$ corresponding to the parameter $\lambda$.

\begin{remark}\label{klgyu} Set $\mathbb R^*:=\mathbb R\setminus\{0\}$ and define a measure $\widetilde \nu$ on $\mathbb R^*$ by
\begin{equation}\label{hjufdtu}\widetilde \nu(dx):=\chi_{\mathbb R^*}(x)\,\frac1{x^2}\,\nu(dx).
\end{equation}
Let also $\varepsilon_0$ denote the Dirac measure at $0$. Then, we can define a unitary isomorphism
\begin{equation}\label{vgdcyd} U:\mathcal G\to\widetilde {\mathcal G}:= L^2(T\times\mathbb R,\sigma\otimes(\nu(\{0\})\varepsilon_0+\widetilde\nu))\end{equation}
by setting
\begin{equation} (Uf)(t,x):=\begin{cases}f(t,0),&\text{if }x=0,\\
xf(t,x),&\text{if }x\ne0.\end{cases}\end{equation}
We can naturally extend $U$ to a unitary isomorphism
\begin{equation}\label{vgcdydy} U:\mathcal F^Q(\mathcal G)\to\mathcal F^Q(\widetilde {\mathcal G}).\end{equation}
Under this isomorphism, each operator $\la f,\xi\ra$ goes over into the operator
\begin{equation} \label{kif8ur8} a^+(f\otimes\chi_{\{0\}})+a^-(f\otimes\chi_{\{0\}})+a^{+}(f\otimes\mu_1)+a^0(f\otimes\mu_1)+a^-(f\otimes\mu_1),\end{equation}
defined on $U\mathfrak D$.
The operator $a^+(f\otimes\chi_{\{0\}})+a^-(f\otimes\chi_{\{0\}})$ gives the $Q$-Gaussian part of the process, the operator
$a^{+}(f\otimes\mu_1)+a^0(f\otimes\mu_1)+a^-(f\otimes\mu_1)$ gives the `jump part' of the process, while $\widetilde\nu$ is the {\it $Q$-L\'evy measure\/} of the process.
\end{remark}

\begin{remark}
It can be shown that each  $F\in\mathfrak D$ is an analytic vector for each operator $\la f,\xi\ra$ with  $f\in B_0(T)$, which implies that the operators $\la f,\xi\ra$  are essentially self-adjoint on $\mathfrak D$. In the case where the measure $\nu$ is compactly supported, this is a trivial fact. In the general case, one has to use  estimate \eqref{cfdh}, and the proof becomes more involved.
\end{remark}

We now introduce the vacuum state $\tau$ on the unital $*$-algebra $\mathscr P$ generated by the operators $(\la f,\xi\ra)_{f\in B_0(T)}$.

\begin{proposition}\label{hjvghfc}
The $n$-th $Q$-cumulant measure of $(\la f,\xi\ra)_{f\in B_0(T)}$ is given by
\begin{gather*}
c_1(dt_1)=0,\notag\\
 c_n(dt_1\times \dots\times dt_n)=\left(\int_{\mathbb R}x^{n-2}\,\nu(dx)\right)
 \delta(dt_1\times\dots\times dt_n),\quad n\ge2.
 \end{gather*}
Hence, $(\la f,\xi\ra)_{f\in B_0(T)}$ is a $Q$-L\'evy process.
\end{proposition}

\begin{remark}
For each $f\in B_0(T)$, we define
$$ C_n(\la f,\xi\ra):=C_n(\la f,\xi\ra,\dots,\la f,\xi\ra)$$
to be the $n$-th $Q$-cumulant of the random variable $\la f,\xi\ra$. Then, by  Proposition~\ref{hjvghfc}, for each $\Delta\in\mathcal B_0(T)$,
$$ C_n(\la\chi_\Delta,\xi\ra)=\left(\int_{\mathbb R}x^{n-2}\,\nu(dx)\right)\sigma(\Delta),\quad n\ge2.$$
Hence, in view of Remark~\ref{klgyu},
$$ C_2(\la\chi_\Delta,\xi\ra)=\sigma(\Delta),\quad C_n(\la\chi_\Delta,\xi\ra)=\left(\int_{\mathbb R^*}x^{n}\,\widetilde\nu(dx)\right)\sigma(\Delta),\quad n\ge3.$$
In particular, if $\sigma(\Delta)=1$, the second $Q$-cumulant of $\la\chi_\Delta,\xi\ra$ is 1, and the $n$-th $Q$-cumulant
($n\ge3$) is equal to the $n$-th moment of the $Q$-L\'evy measure. In the classical case, $Q\equiv1$, this property is equivalent to the infinite divisibility of the distribution of a random variable, see e.g.\ \cite{Shiryaev}. We also refer the reader to Nica and Speicher \cite{NicaSpeicher} and to Anshelevich \cite{A2}, where a similar property was discussed in the framework of free probability and in the case of $q$-commutation relations $(-1< q<1)$, respectively.

\end{remark}

\begin{proof}[Proof of Proposition~\ref{hjvghfc}] It suffices to show that, for any $f_1,\dots,f_n\in B_0(T)$,
\begin{align}
&\tau(\la f_1,\xi\ra\dotsm\la f_n,\xi\ra)\notag\\
&\quad=\sum_{\mathcal V\in \mathcal P^{(n)}_{\ge2}}\int_{T^n}f_1(t_1)\dotsm f_n(t_n)
Q(\mathcal V;t_1,\dots,t_n) \prod_{B\in\mathcal V}\int_{\mathbb R}x^{|B|-2}\,\nu(dx)\, \delta(dt_B).\label{gyuf}
\end{align}
If $\nu(\{0\})=0$, then formula \eqref{gyuf} immediately follows  from Corollary~\ref{fyde} and Remark~\ref{klgyu}.
In the general case, one may argue as follows. Noting that $\nu$ is a probability measure on $\mathbb R$, we get the following representation:
\begin{align}
&\tau(\la f_1,\xi\ra\dotsm\la f_n,\xi\ra)
=\int_{(T\times \mathbb R)^n}\sigma(dt_1)\nu(dx_1)\dotsm \sigma(dt_n)\nu(dx_n)f_1(t_1)\dotsm f_n(t_n)\notag\\
&\quad
\times
\big(
(\partial_{(t_1,x_1)}^\dag+x_1 n(t_1,x_1)+\partial_{(t_1,x_1)})\dotsm (\partial_{(t_n,x_n)}^\dag+x_n n(t_n,x_n)+\partial_{(t_n,x_n)})\Omega,\Omega
\big)_{\mathcal F^Q(\h)},\label{fdtes}
\end{align}
where $n(t,x):=\partial^\dag_{(t,x)}\partial_{(t,x)}$ is the neutral operator at point $(t,x)$. Expand the product in  the second line
of formula \eqref{fdtes}, and leave only those terms which are not {\it a priori} equal to zero. Now  formula \eqref{gyuf} easily follows if we use the following interpretation of partitions $\mathcal V\in  \mathcal P^{(n)}_{\ge2}$. Each $\mathcal V$ corresponds to the term which has the following structure. For each block $B=\{i_1,\dots,i_k\}\in \mathcal V$ with $i_1<i_2<\dots<i_k$, we have:
at place $i_k$ there
 is a creation operator; then at  places $i_{k-1}, i_{k-2},\dots,i_2$ there are neutral operators which act on  place $i_k$ (i.e., they identify their variables with  $(t_k,x_k)$), and finally at place $i_1$ there is an annihilation operator which annihilates place $i_k$ (i.e., variable  $(t_k,x_k)$). To reach place $i_k$, the annihilation operator has to cross all variables $(t_j,x_j)$ with $i_1<j<i_k$ which have not yet been killed, i.e., each $j$ is the maximal point of a block $B'\in\mathcal V$ such that the minimal point of $B'$ is smaller than $i_1=\min B$. These crossings yield the corresponding $Q$-functions.
\end{proof}

We will now show that the  $Q$-L\'evy processes we have just constructed possess a property of pyramidal independence. 
The latter notion was  introduced by K\"ummerer (in an unpublished preprint) and by Bo\.zejko and Speicher in  \cite{BS_1996}. We also refer the reader to Lehner \cite[subsec.~3.5]{Lehner1} for some consequences of pyramidal independence, and to Anshelevich  \cite[Lemma~3.3]{A2} for a discussion of pyramidal independence of increments of a $q$-L\'evy process for $-1<q<1$.

\begin{proposition}\label{futfr}
Let $A,B\in\mathcal B(T)$, $A\cap B=\varnothing$, and let $f_1,\dots,f_m,f_{m+1},\dots,f_{m+k},\linebreak g_1,\dots,g_n\in B_0(T)$ be such that
$\operatorname{supp}f_i\subset A$, $i=1,\dots,m+k$, $\operatorname{supp} g_j\subset B$, $j=1,\dots,n$.
Then
\begin{multline}
\tau\big(\la f_1,\xi\ra\dotsm \la f_m,\xi\ra \la g_1,\xi\ra \dotsm\la g_n,\xi\ra \la f_{m+1},\xi\ra\dotsm \la f_{m+k},\xi\ra\big)\\
=\tau\big(\la f_1,\xi\ra\dotsm \la f_m,\xi\ra  \la f_{m+1},\xi\ra\dotsm \la f_{m+k},\xi\ra\big)
\tau\big( \la g_1,\xi\ra \dotsm\la g_n,\xi\ra \big).\label{gdrtek}
\end{multline}
\end{proposition}

\begin{proof} Write the left hand side of \eqref{gdrtek} as
\begin{equation}\label{hgfdytd} \big(\la g_1,\xi\ra \dotsm\la g_n,\xi\ra \la f_{m+1},\xi\ra\dotsm \la f_{m+k},\xi\ra\Omega,
\la f_m,\xi\ra\dotsm \la f_1,\xi\ra\Omega\big)_{\mathcal F^Q(\mathcal G)}.
\end{equation}
Observe that both $\la f_{m+1},\xi\ra\dotsm \la f_{m+k},\xi\ra\Omega$ and $\la f_m,\xi\ra\dotsm \la f_1,\xi\ra\Omega$ belong to the subspace $\mathcal F^Q (L^2(A\times\mathbb R,\sigma\otimes\nu))$ of $\mathcal F^Q(\mathcal G)$. Furthermore, it is easy to see that,
for each $g_i$ and any $F\in\mathcal F^Q(L^2(A\times\mathbb R,\sigma\otimes\nu))\cap\mathfrak D$ and $G\in\mathcal F^Q(L^2(B\times\mathbb R,\sigma\otimes\nu))\cap\mathfrak D$,
$$ \la g_i,\xi\ra(G\cd F)= (\la g_i,\xi\ra G)\cd F.$$
Therefore, the expression in \eqref{hgfdytd} is equal to
$$\big(\big(\la g_1,\xi\ra \dotsm\la g_n,\xi\ra\Omega\big)\cd \big(\la f_{m+1},\xi\ra\dotsm \la f_{m+k},\xi\ra\Omega\big),
\la f_m,\xi\ra\dotsm \la f_1,\xi\ra\Omega\big)_{\mathcal F^Q(\mathcal G)}.$$
But for any $F_1,F_2\in\mathcal F^Q(L^2(A\times\mathbb R,\sigma\otimes\nu))$ and $G\in\mathcal F^Q(L^2(B\times\mathbb R,\sigma\otimes\nu))$,
$$ (G\cd F_1,F_2)_{\mathcal F^Q(\mathcal G)}=(G,\Omega)_{\mathcal F^Q(\mathcal G)}  (F_1,F_2)_{\mathcal F^Q(\mathcal G)},$$
from where \eqref{gdrtek} follows.
\end{proof}

Analogously to Section~\ref{cfydry}, we may now introduce a noncommutative space $L^2(\tau)$. Furthermore, Proposition~\ref{drawjh} allows an extension to the L\'evy case.

\begin{proposition}\label{fhysdyts} \rom{(i)} The vacuum vector $\Omega$ is cyclic for the family of operators\linebreak $(\la f,\xi\ra)_{f\in B_0(T)}$.

\rom{(ii)} Recall that $\mathscr P$ denotes the
 unital algebra  generated by the operators $(\la f,\xi\ra)_{f\in B_0(T)}$\,, and let $\mathscr P_0$ be defined as before.
 Consider a linear mapping $I:\mathscr P\to\mathcal F^Q(\mathcal G)$ defined by $Ip:=p\Omega$ for $p\in\mathscr P$. Then $Ip$ does not depend on the choice of $p\in\mathscr P/\mathscr P_0$ and $I$ extends to a unitary operator $I:L^2(\tau)\to\mathcal F^Q(\mathcal G)$.
\end{proposition}

\begin{proof} Clearly, we only need to prove part i).
Denote by $\mathscr U$ the closure of the set $\mathscr P\Omega$ in  $\mathcal F^Q(\mathcal G)$.
To prove the proposition, it suffices to show  that $\mathscr U=\mathcal F^Q(\mathcal G)$.
In view of  assumption \eqref{dstrst}, the set of functions $$\{f(t)x^k  \mid f\in B_0(T),\ k\in\mathbb Z_+\}$$ is total in $\mathcal G$
(i.e., its closed linear span coincides with  $\mathcal G$). Therefore, the set
\begin{equation}\label{ftufytufd}\big\{\Omega,\, P_i\big[f^{(i)}(t_1,\dots,t_i)x_1^{l_1}\dotsm x_i^{l_i}\big]\mid
f^{(i)}\in B_0(T^i\mapsto\mathbb C),\, (l_1,\dots,l_i)\in\mathbb Z_+^i,\, i\in\mathbb N
\big\}\end{equation}
is total in $\mathcal F^Q(\mathcal G)$. Hence, it suffices to show that, for any multi-index $ (l_1,\dots,l_i)\in\mathbb Z_+^i$ with $i\in\mathbb N$,
\begin{equation}\label{hfyut}
\big\{ P_i\big[f^{(i)}(t_1,\dots,t_i)x_1^{l_1}\dotsm x_i^{l_i}\big]\mid
f^{(i)}\in B_0(T^i\mapsto\mathbb C)\big\}\subset \mathscr U.
\end{equation}
 We will prove \eqref{hfyut} by induction on $l_1+\dots+l_i+i$. The statement trivially holds when this number is 1. Let us assume that the statement holds for $1,2,\dots,n$, and let us prove it for $n+1$. So, we fix any multi-index $(l_1,\dots,l_i)$ such that $l_1+\dots+l_i+i=n+1$. Since the measure $\sigma$ is non-atomic, it suffices to show that, for any mutually disjoint sets $\Delta_1,\dots, \Delta_i\in \mathcal B_0(T)$, we have the inclusion
$$P_i\big[\chi_{\Delta_1}(t_1)\dotsm\chi_{\Delta_i}(t_i)x_1^{l_1}\dotsm x_i^{l_i}\big]=\big((\chi_{\Delta_1}\otimes \mu_{l_1})
\cd\dotsm\cd (\chi_{\Delta_i}\otimes\mu_{l_i})\big)(t_1,x_1,\dots,t_i,x_i)\in\mathscr U.
$$
(Recall notation \eqref{hcdydyd}.) We have to distinguish two cases.

Case 1: $l_1=0$. Then, by Proposition~\ref{ctst} and formula \eqref{dd},
$$ (\chi_{\Delta_1}\otimes \mu_0)
\cd (\chi_{\Delta_2}\otimes \mu_{l_2})\cd\dotsm\cd (\chi_{\Delta_i}\otimes\mu_{l_i})=\la \chi_{\Delta_1},\xi\ra \big((\chi_{\Delta_2}\otimes \mu_{l_2})\cd\dotsm\cd (\chi_{\Delta_i}\otimes\mu_{l_i})\big), $$
and the statement follows by the assumption of induction.

Case 2: $l_1\ge1$. Then, again using Proposition~\ref{ctst} and formula \eqref{dd},
\begin{align*}
&(\chi_{\Delta_1}\otimes \mu_{l_1})
\cd\dotsm\cd (\chi_{\Delta_i}\otimes\mu_{l_i})=\la \chi_{\Delta_1},\xi\ra \big((\chi_{\Delta_1}\otimes \mu_{l_1-1})
\cd\dotsm\cd (\chi_{\Delta_i}\otimes\mu_{l_i})\big)\\
&\quad- (\chi_{\Delta_1}\otimes 1)\cd (\chi_{\Delta_1}\otimes \mu_{l_1-1})
\cd\dotsm\cd (\chi_{\Delta_i}\otimes\mu_{l_i})\\
&\quad-\sigma(\Delta_1)\int_{\mathbb R}x^{l_1-1}\,\nu(dx)(\chi_{\Delta_2}\otimes \mu_{l_2})\cd\dotsm\cd (\chi_{\Delta_i}\otimes\mu_{l_i}),
\end{align*}
and the statement again follows by the assumption of induction.
\end{proof}

\section{Nualart--Schoutens-type chaotic decomposition for\newline $Q$-L\'evy processes}\label{kifr78ufsrs}

Our aim now is to derive a counterpart of the Nualart--Schoutens chaotic decomposition \cite{NS} for $Q$-L\'evy processes. By taking `powers of the jumps', we obtain the sequence of power jump processes
\begin{equation}\label{ftyyded} X_k(f):=a^+(f\otimes\mu_{k-1})+a^0(f\otimes\mu_k)+a^-(f\otimes\mu_{k-1}),\quad f\in B_0(T),\ k\in\mathbb N.\end{equation}
In particular, $X_{1}(f)=\la f,\xi\ra$. (All these operators map the domain $\mathfrak D$ into itself.)

\begin{remark} Note that under the unitary isomorphism $U$ defined by \eqref{hjufdtu}--\eqref{vgcdydy}, the operator $X_k(f)$ with $k\ge2$, goes over into the operator
\begin{equation}\label{ghcydyr}a^+(f\otimes\mu_{k})+a^0(f\otimes\mu_k)+a^-(f\otimes\mu_{k}),\end{equation}
 compare with formula \eqref{kif8ur8} which gives the image of $\la f,\xi\ra=X_1(f)$. In formula \eqref{ghcydyr},   $\mu_k(x)=x^k$
 can be interpreted as the $k$-th power of the `jump' $x$.
\end{remark}

For a fixed $f\in B_0(T)$, we now orthogonalize the noncommutative random variables $\big(X^{(k)}(f)\big)_{k=1}^\infty$ in $L^2(\tau)$. Noting that
\begin{equation}\label{uyrfr}(X_{k}(f)\Omega)(t,x)=f(t)x^{k-1},\quad k\in\mathbb N,\end{equation}
this is equivalent to the procedure of orthogonalization of the monomials $(x^k)_{k=0}^\infty$ in $L^2(\mathbb R,\nu)$.

Let $(p^{(k)})_{k=0}^\infty$ denote the system of monic orthogonal polynomials in $L^2(\mathbb R,\nu)$. (If the support of $\nu$ is finite and consists of $N$ points, we set $p^{(k)}:=0$ for $k\ge N$.) By Favard's theorem (see e.g.\ 
\cite[Ch.~I, Sec.~4]{Chihara}), we have the recursive formula
\begin{equation}\label{hdtrss}xp^{(k)}(x)=p^{(k+1)}(x)+b_kp^{(k)}(x)+a_kp^{(k-1)}(x),\quad k\in\mathbb Z_+,\end{equation}
with $p^{(-1)}(x):=0$, $a_k>0$, and $b_k\in\mathbb R$. (If the support of $\nu$ has $N$ points, $a_k=0$ for $k\ge N$.)
Thus, by virtue of \eqref{ftyyded}--\eqref{hdtrss}, the orthogonalized power jumps processes are
$$ Y_k(f):=a^+(f\otimes p^{(k)})+a^0\big(f\otimes(p^{(k+1)}+b_{k}p^{(k)}+a_{k}p^{(k-1)})\big)+a^-(f\otimes p^{(k)}),$$
where $f\in B_0(T)$ and  $k\in\mathbb Z_+$.
(It is convenient for us to start the numeration of the $Y$-processes from 0, rather than from 1.) For $\Delta\in\mathcal B_0(T)$, we will also denote $Y_k(\Delta):=Y_k(\chi_\Delta)$.

For each multi-index $(k_1,\dots,k_n)\in\mathbb Z_+^n$ and each function $f^{(n)}\in\hc^{\otimes n}$, we can now construct a noncommutative multiple stochastic integral
\begin{equation}\label{vftdtu}\int_{T^n}f^{(n)}(t_1,\dots,t_n)\,Y_{k_1}(dt_1)\dotsm Y_{k_n}(dt_n)\in L^2(\tau)\end{equation} as follows.
We first choose arbitrary $\Delta_1,\dots,\Delta_n\in \mathcal B_0(T)$, mutually disjoint, and define
$$\int_{T^n}\chi_{\Delta_1}(t_1)\dotsm\chi_{\Delta_n}(t_n)\,Y_{k_1}(dt_1)\dotsm Y_{k_n}(dt_n):=Y_{k_1}(\Delta_1)\dotsm Y_{k_n}(\Delta_n).$$
Since  $\Delta_1,\dots,\Delta_n$ are mutually disjoint, we have
$$Y_{k_1}(\Delta_1)\dotsm Y_{k_n}(\Delta_n)\Omega=(\chi_{\Delta_1}\otimes p^{(k_1)})\cd\dotsm\cd (\chi_{\Delta_n}\otimes p^{(k_n)}).$$
Since the measure $\sigma$ is non-atomic, the functions $\chi_{\Delta_1}\otimes\dots\otimes \chi_{\Delta_n}$ with $\Delta_1,\dots,\Delta_n$ as above form a total set in $\hc^{\otimes n}$. Thus, by linearity and continuity the definition of a multiple stochastic integral is extendable to the whole of  $\hc^{\otimes n}$. Thus, under the unitary isomorphism $I:L^2(\tau)\to\mathcal F^Q(\mathcal G)$ from Proposition~\ref{fhysdyts}, the image of the multiple stochastic integral in \eqref{vftdtu}
is $P_n\big[f^{(n)}(t_1,\dots,t_n)p^{(k_1)}(x_1)\dotsm p^{(k_n)}(x_n)\big]$. Denote by $\mathcal F_{(k_1,\dots,k_n)}$ the subspace of $\mathcal F^Q(\mathcal G)$ consisting of all such elements. (In fact, $\mathcal F_{(k_1,\dots,k_n)}$ is a subspace of $\mathcal G_{\mathbb C}^{\cd n}$.) In view of the $Q$-symmetry, for each permutation $\pi\in S_n$, the spaces $\mathcal F_{(k_1,\dots,k_n)}$ and $\mathcal F_{(k_{\pi(1)},\dots,k_{\pi(n)})}$ coincide. Thus we can always assume that $k_1\le k_2\le\dots\le k_n$.
In view of this, we will use the following notation. Denote by $\mathbb Z_{+,\,\mathrm{fin}}^\infty$ the set of all infinite sequences $\alpha=(\alpha_0,\alpha_1,\alpha_2,\dots)\in\mathbb Z_+^\infty$ such that only a finite number of $\alpha_j$'s are not equal to zero. Let $|\alpha|:=\alpha_0+\alpha_1+\alpha_2+\dotsm$. For each $\alpha\in \mathbb Z_{+,\,\mathrm{fin}}^\infty$, we denote by $\mathcal F_\alpha$ the space $\mathcal F_{(k_1,\dots,k_n)}$ with  $n=|\alpha|$ and
$k_1=\dots=k_{\alpha_0}=0$, $k_{\alpha_0+1}=\dots =k_{\alpha_0+\alpha_1}=1$, $k_{\alpha_0+\alpha_1+1}=\dots=k_{\alpha_0+\alpha_1+\alpha_2}=2$, and so on. (In the case where  $\alpha=(0,0,\dots)$,  $\mathcal F_\alpha$ will mean the vacuum space.)

Using the orthogonality of the polynomials $(p^{(k)})_{k=0}^\infty$ in $L^2(\mathbb R,\nu)$, we easily conclude from Proposition~\ref{cftrst} that, for different multi-indices $\alpha,\beta\in \mathbb Z_{+,\,\mathrm{fin}}^\infty$, the spaces $\mathcal F_\alpha$ and $\mathcal F_\beta$ are orthogonal in $\mathcal F^Q(\mathcal G)$. Since the polynomials are dense in $L^2(\mathbb R,\nu)$, we therefore conclude that $\mathcal F^Q(\mathcal G)=\bigoplus_{\alpha\in \mathbb Z_{+,\,\mathrm{fin}}^\infty}\mathcal F_\alpha$.

We next note that, for $\alpha\in \mathbb Z_{+,\,\mathrm{fin}}^\infty$, a general element of $\mathcal F_\alpha$ has the form
 \begin{equation}\label{cgtrdtr}P_{|\alpha|}\big[f^{(|\alpha|)}(t_1,\dots,t_{|\alpha|})p^{(0)}(x_1)\dotsm p^{(0)}(x_{\alpha_0})p^{(1)}(x_{\alpha_0+1})\dotsm p^{(1)}(x_{\alpha_0+\alpha_1})
\dotsm\big],\end{equation}
with $f^{(|\alpha|)}\in\hc^{\otimes |\alpha|}$.
Using Proposition~\ref{jvfytdy}, we have the following identity for the $Q$-symmetrization operators:
$$ P_{|\alpha|}=P_{|\alpha|}(P_{\alpha_0}\otimes P_{\alpha_1}\otimes P_{\alpha_2}\otimes\dotsm),$$
where we set $P_0:=\mathbf 1$. Therefore, without loss of generality, we may assume that a general element of $\mathcal F_\alpha$ is given by the formula  \eqref{cgtrdtr} in which
$
f^{(|\alpha|)}\in \hc^{\cd\alpha_0}\otimes \hc^{\cd \alpha_1}\otimes\hc^{\cd\alpha_2}\otimes\dotsm$.

For each $\alpha\in \mathbb Z_{+,\,\mathrm{fin}}^\infty$, we now define a complex Hilbert space
\begin{equation}\label{ggytdyd}\mathbb F_\alpha:= \hc^{\cd\alpha_0}\otimes \hc^{\cd \alpha_1}\otimes\hc^{\cd\alpha_2}\otimes\dotsm\,\left(\prod_{i\ge0}\alpha_i!\, C_i
^{\alpha_i}\right).\end{equation}
Here, for $i\ge0$, $C_i:=\int_{\mathbb R}|p^{(i)}(x)|^2\,\nu(dx)$. Recall that, for a Hilbert space $\mathscr H$ and a constant $C>0$, we denote by $\mathscr H C$ the Hilbert space which coincides with $\mathscr H$ as a set and which satisfies  $\|\cdot\|^2_{\mathscr H C}:=\|\cdot\|_{\mathscr H}^2C$.

Using again the orthogonality of the polynomials $(p^{(k)})_{k=0}^\infty$ in $L^2(\mathbb R,\nu)$ and Proposition~\ref{cftrst},
we see that, for each $f^{(|\alpha|)}\in\mathbb F_\alpha$,  the square of the $\mathcal F^Q(\mathcal G)$-norm of the
expression in \eqref{cgtrdtr} is equal to
$ \|f^{(|\alpha|)}\|^2_{\mathbb F_\alpha}$. Thus, we have proven the following

\begin{theorem} For each  $Q$-L\'evy process constructed in Section~\rom{\ref{kigf7ur7}}, the following unitary operator gives an orthogonal expansion of $L^2(\tau)$ in noncommutative multiple stochastic integrals:
\begin{multline*}\bigoplus_{\alpha\in \mathbb Z_{+,\,\mathrm{fin}}^\infty}\mathbb F_\alpha\ni (f_\alpha)_{\alpha\in\mathbb Z_{+,\,\mathrm{fin}}^\infty}\mapsto \sum_{\alpha\in \mathbb Z_{+,\,\mathrm{fin}}^\infty}\int_{T^{|\alpha|}}f_\alpha(t_1,\dots,t_{|\alpha|}) Y_0(dt_1)\dotsm Y_0(dt_{\alpha_0})\\
\times
Y_1(dt_{\alpha_0+1})\dotsm Y_1(dt_{\alpha_0+\alpha_1})\dotsm\in L^2(\tau),
 \end{multline*}
where the spaces $\mathbb F_\alpha$ are given by \eqref{ggytdyd}.
\end{theorem}

\begin{center}
{\bf Acknowledgements}\end{center} We would  like to thank the referee for a careful reading of the
manuscript and making very useful comments and suggestions. 
 The research was partially supported by the International Joint Project grant 2008/R2 of the Royal Society.
 MB and EL  acknowledge the financial support of the SFB~701 ``Spectral structures and topological methods in mathematics'', Bielefeld University. MB and JW were partially supported by the
 the Polish MNiSW grant NN201 364436.

\end{document}